\documentclass[a4paper,10pt]{amsart}
\usepackage[matrix,arrow,tips,curve]{xy}
\usepackage[english]{babel} 
\usepackage[leqno]{amsmath}
\usepackage{amssymb,amsthm}
\usepackage{amscd}
\usepackage{enumerate}

\newtheorem{thm}{Theorem}[subsection]
\newtheorem{lemma}[thm]{Lemma}

\newtheorem{prop}[thm]{Proposition}
\newtheorem{add}[thm]{Addendum}

\newtheorem{cor}[thm]{Corollary}

\theoremstyle{remark}
\newtheorem{fact}[thm]{Fact}
\newtheorem{remark}[thm]{Remark}
\theoremstyle{definition}
\newtheorem{defi}[thm]{Definition}
\newtheorem{nota}[thm]{}

\newtheorem{example}[thm]{Example}

\newcommand{\la}{\longrightarrow}

\newcommand{\ov}{\overline}

\newcommand{\Div}{\operatorname{Div}}

\newcommand{\Spec}{\operatorname{Spec}}

\newcommand{\Pic}{\operatorname{Pic}}

\newcommand{\mdeg}{\operatorname{{\underline{de}g}}}

\def\B{\mathcal B}

\def\L{\mathcal L}

\def\O{\mathcal O}

\def\X{\mathcal X}

\newcommand{\Z}{\mathbb{Z}}

\def\md{\underline{d}}
\def\mg{\underline{g}}

\def\mo{\underline{0}}

\newcommand{\g}{\gamma}

\newcommand{\dd}{\delta}

\newcommand{\pr}[1]{\mathbb{P}^{#1}}

\newcommand{\Xn}{X^{\nu}}

\newcommand{\Mgb}{\ov{M}_g}
\newcommand{\Hgb}{\ov{H}_g}

\def\PXb{\overline{P^d_X}}

\def\FM{F_{M}(X)}
\def\FL{F_{(L_1,L_2)}(X)}

\newcommand{\picX}[1]{\Pic^{#1}X}

\newcommand{\sing}{X_{\text{sing}}}
\newcommand{\sep}{X_{\text{sep}}}

\def\BXd{B_d  (X)}
\def\BX1{B_1  (X)}
\def\BXt{B_2  (X)}

\newcommand{\df}{\frac{\delta}{2}}
\newcommand{\dfZ}{\frac{\delta_Z}{2}}
\newcommand{\Zd}{Z^-_{\md}}
\newcommand{\Ys}{Y_s}
\newcommand{\Ws}{W_s}
\newcommand{\Cl}{\operatorname{Cliff}}

\begin{document}
\begin{center}
{\bf\large Linear series on semistable curves.}

\noindent {{Lucia  Caporaso}}\footnote{Dipartimento di Matematica, Universit\`a Roma Tre,
Largo S.L.Murialdo,  00146 Roma Italy - caporaso@mat.uniroma3.it}\

\today
\end{center}

\noindent {\it Abstract.} For a semistable curve $X$ of genus $g$,
the number    $h^0(X,L)$ is studied for  line bundles $L$ of degree $d$  parametrized by the compactified Picard scheme. The theorem of Riemann is shown to hold.
The theorem of Clifford is shown to hold  in the following cases: $X$ has two components;
$X$ is any semistable curve and $d=0$ or $d=2g-2$; $X$ is stable, free from separating nodes, and $d\leq 4$.  These results are shown to be sharp.
Applications to the Clifford index, to the combinatorial description of hyperelliptic curves, and 
to plane quintics are given.
\tableofcontents
\section{Introduction and preliminaries}

The dimension of   complete linear series on singular curves is, in general, quite difficult to control.
This is one of the reasons why several interesting degeneration problems about line bundles and linear series remain
unsolved.
For singular curves the Riemann-Roch theorem  does not yield
as strong information  as for smooth curves, and several other classical theorems   fail, as we shall illustrate.

On the other hand, it is well known that the Picard scheme of a singular curve
tends to be too large, so that   any good compactification of the
generalized jacobian   parametrizes only  
a distinguished subset of line bundles.
At  present time the geometric and functorial properties of the  compactified Picard scheme
are rather  well understood,   making
it   a  natural place to study limits of line bundles and related problems.

This is the main theme of this paper, which  investigates the dimension of complete linear series
parametrized by the compactified Picard scheme of stable curves. They correspond to so-called balanced line bundles on semistable curves (defined in \ref{bal}).

There exist other approaches to   this type of questions. Some of them are by now considered classical, like the theory 
of admissible covers, of J.Harris and D.Mumford (\cite{HM}), and the theory
of limit linear series, of D.Eisenbud and J.Harris (\cite{EH}).
Although these techniques have been successfully applied  by their creators to solve important problems, and they
have been further studied by others (\cite{bruno}, \cite{EM}, \cite{oss} for example), several open questions, some considered in the present paper, remain  open.
Our method,   applied also in \cite{Cbin},
 is different as it departs from the compactified Picard scheme and
does not use degeneration techniques.

We proceed in analogy with the classical theory of Riemann surfaces.
Our first result is Theorem~\ref{sr}, generalizing   a theorem of Riemann,
computing $h^0(X,L)$ for a balanced line bundle $L$ of large degree on a semistable curve $X$.
Although this theorem fails on infinitely many components of the Picard  scheme
of a reducible curve (see Example~\ref{Pr}), we prove that,
 quite
pleasingly, it does   hold for every balanced  line bundle, that is for every element of
the compactified Picard scheme of $X$.

We then turn to study the theorem of Clifford.  The situation is  much more complex, as this theorem turns out to fail, even for balanced line bundles,
in certain situations.
Nonetheless, we prove that Clifford's theorem  does hold in several cases. Namely,  
 it holds for all degrees on curves with two components (Theorem~\ref{clvine}).
Also it holds for all stable curves if the degree is $0$ or $2g-2$ (Theorems~\ref{cc} and \ref{c0}).
Finally,    it
holds for degree at most $4$, for all  stable curves free from separating nodes (Theorem~\ref{cl4}).  Some counterexamples are exhibited to
 show  that the result is sharp:
  the Clifford inequality fails for all positive degrees  for curves with separating nodes; furthermore
 if $d\geq 5$ then it fails even for curves free from separating nodes (see Example~\ref{cl5no}).

The last section 
is devoted to applications. For curves with two components the Clifford's theorem is valid,
 it is thus interesting to study
  their (suitably defined) Clifford's index and its connection with the gonality; we do that in Proposition~\ref{clwh}, stating that a curve is weakly hyperelliptic 
  (i.e. it admits a  balanced ${\rm {g}}^1_2$) if and only if its Clifford index is $0$.
 Next,   we focus on   weakly hyperelliptic curves, give a combinatorial characterization of them (Theorem~\ref{c2=}) and use it to describe the combinatorics of hyperelliptic curves (Proposition~\ref{hypcomb}).
     We conclude the paper with a classification of ${\rm {g}}^2_5$'s on 
 two-component curves of genus $6$ (Theorem~\ref{g25}).
 
\

\noindent
{\it Acknowledgements.} I wish to thank Edoardo Sernesi for 
 several enlightening conversations,    Edoardo Ballico and Silvia Brannetti for some precious remarks.
I am very  grateful to the referees for their careful reports correcting several inaccuracies.
 
\begin{nota}{\it Conventions.}
\label{not}
We work over any algebraically closed field.
The following notation and terminology will be used thoughout the paper.
The word ``curve" stands for reduced projective scheme of pure dimension one.
$X$ is a connected curve, having at most nodes as singularities.  
$g$ is the  arithmetic genus of $X$.
The irreducible component decomposition of $X$ is written $X=\cup_{i=1}^{\g}C_i$,
and $g_i$ is the arithmetic genus of $C_i$. 
We shall usually denote by $Z$ a (complete, reduced, of pure dimension one) subcurve of $X$, by $g_Z$ its arithmetic genus, and by
$Z^c=\ov{X\smallsetminus Z}$ its complementary curve.

Given a line bundle $L\in \Pic X$ we denote by $L_Z$ its restriction to a subcurve $Z$ of $X$.

Given two subcurves $Z, Z'$ of $X$ with no components in common, we shall denote 
\begin{equation}
\label{ZZ}
Z\cdot Z':= \#Z\cap Z'\   \  \  \text{ and }\  \  \  \dd_Z:=Z\cdot Z^c=\#Z\cap Z^c.
\end{equation}
 The formula $g=g_Z+g_{Z^c}+\dd_Z -1$ will be used several times.

Whenever we shall  decompose a curve as a union of subcurves, 
e.g. $X=Z\cup Y$, it will always be understood
that  $Z$ and $Y$   have no components in common.

$\md=(d_1,\ldots, d_{\g})$ will always be an element of $\Z^{\gamma}$, and $|\md|=\sum_1^{\g}d_i$.
By $\md \leq 0$ (resp. $\md \geq 0$) we mean that  $d_i\leq 0$ (resp. $d_i\geq 0$)
for every $i$. We denote by $\picX{\md}$,
  the set of line bundles $L$ on $X$ having multidegree $d_i=\deg_{C_i}L$
for $i=1\ldots \gamma$, and, 
for any integer $r\geq 0$ we set
$W_{\md}^r(X):=\{L\in \Pic ^{\md}X: h^0(L)\geq r+1\}$.\end{nota}

\subsection{Gluing global sections}
In this subsection, we collect several technical lemmas needed in the sequel.
\begin{nota}
\label{}
Let $\nu:Y\to X$ be some partial (possibly total)  normalization of $X$; consider the (surjective) morphism
$\nu^*:\Pic X\to \Pic Y$.
For every $M\in \Pic Y$ we will denote the fiber of $\nu^*$ over $M$  as follows
\begin{equation}
\label{FM}
\FM :=\{L\in \Pic X:\nu^*L=M\}.
\end{equation}

Let $\dd$ be the number of nodes normalized by $\nu:Y\to X$.
For each of such node, $n_i$, let $\{p_i, q_i\}=\nu^{-1}(n_i)$ be its branches. 
We represent  the above data by the self explanatory notation
\begin{equation}
\label{glu}
Y\la X=Y/_{\{p_i=q_i,\  \  i=1,\ldots,  \dd \}}.
\end{equation}
Fix  $M\in\Pic Y$ such that $h^0(Y,M)\neq 0$. Pick  $L\in \FM$; then (cf. \cite{Ctheta} 2.1.1)
\begin{equation}
\label{eleq}
h^0(Y,M)-\dd\leq h^0(X,L)\leq h^0(Y,M).
\end{equation}
To study when  $h^0(X,L)=h^0(Y,M)$
we introduce a convenient notation.
\end{nota}

\begin{defi}
\label{np}
Let $Y$ be a  curve, $M\in \Pic Y$  and $p,q$ nonsingular points of $Y$.
We say that $p$ and $q$ are a neutral pair of $M$,
and write
$
p\sim_M q, 
$
if \begin{equation}
\label{CN}
h^0(Y,M-p)=h^0(Y,M-q)=h^0(Y,M-p-q).
\end{equation} 
\end{defi}

\begin{remark}
\label{npr}
Notation as in \ref{np}. 
\begin{enumerate}[(A)]
\item
\label{npe}
The relation  $p\sim_Mq$ is an equivalence relation. 
\item
\label{npb}
 If $p$ and $q$ lie in different connected components of $Y$,   $p\sim_Mq$ if and only if $p$ and $q$ are base points of $M$. 
\item
\label{np0}
$
p\sim_{\O_Y} q 
$  if and only if $p$ and $q$ lie in the same connected component of $Y$. 
\item
\label{np1}
If $M$ is very ample, then
 $M$ has no neutral pair . 
\end{enumerate}
\end{remark}
\begin{lemma}
\label{bpl} Let 
$ 
Y=Z_1\coprod Z_2/_{\{p_i=q_i, \  i=1,\ldots, \beta\}},
$ 
where $Z_1$ and $Z_2$ are two nodal curves, and 
$p_1,\ldots, p_{\beta}$ 
(respectively $q_1,\ldots, q_{\beta}$)  smooth points of $Z_1$
(resp. of $Z_2$).
Let $M \in \Pic Y$ and let $p\in Z_1$, $q\in Z_2$ be smooth points of $Y$.
If 
$
p\sim _{M}q$ then $p$ is a base point of $M_{Z_1}(-\sum_{i=1}^{\beta}p_i)$ (and 
$q$ is a base point of $M_{Z_2}(-\sum_{i=1}^{\beta}q_i)$).
\end{lemma}
\begin{proof}
Suppose   that $p $ is not a base point of  
$M_{Z_1}(-\sum_{i=1}^{\beta}p_i)$. Then there exists $s_1\in H^0(Z_1,M_{Z_1}(-\sum_{i=1}^{\beta}p_i))$
such that $s_1(p)\neq 0$. Since $s_1$ vanishes at $p_i$ for $i\leq \beta$, $s_1$ can be glued to
the zero section in $H^0(Z_2, M_{Z_2})$, to give a section $s\in H^0(Y,M)$.
By construction,
$s(p)\neq 0$
and $s(q)=0$. Therefore $p\not\sim _{M}q$.  
\end{proof}

The next Lemma follows trivially from Lemmas 2.2.3 and  2.2.4 in \cite{Ctheta}.
\begin{lemma}
\label{d1l} Let $Y$ be a  nodal curve,  $p$ and $q$ two nonsingular points of $Y$ and  $Y\to
X=Y/_{\{p=q\}}$. 
Let $M\in \Pic Y$  be such that  $h^0(Y,M)\neq 0$.  

There exists $L\in \FM$ such that
$h^0(X,L)=h^0(Y,M)$ if and only if 
$ 
p\sim_M q.
$ 
If  $Y$ is connected, such an $L$   is unique (if it exists) if and only if 
$p$ and $q$ are not base points for $M$.
   \end{lemma}

\begin{lemma}
\label{d2} Let 
$ 
Y=Z_1\coprod Z_2\to X=Y/_{\{p_i=q_i, \  i=1,\ldots, \dd\}},
$ 
where $p_1,\ldots, p_{\dd}$ 
(respectively $q_1,\ldots, q_{\dd}$)  are non singular points of $Z_1$
(resp. of $Z_2$).
Let $M=(M_1,M_2)\in \Pic Z_1\times \Pic Z_2=\Pic Y$; assume   $h^0(Y, M)\geq 2$, and $p_i\not\sim_Mq_i$ $\forall
i$. Then there exists $L\in F_M(X)$ such that $h^0(X,L)= h^0(Y, M)-1$
if and only if
$$
p_i\sim _{M_1}p_j \  \text{ and }\  q_i\sim _{M_2}q_j,\  \  \forall i,j.
$$
\end{lemma}
\begin{proof}
If $\dd=1$ we have $F_M(X)=\{L\}$ and our assumption $p_1\not\sim_Mq_1$ implies,
by Lemma~\ref{d1l}, that $h^0(X,L)= h^0(Y, M)-1$. From now on we let  $\dd \geq 2$.
 Assume first $\dd=2$.
 Denote 
$ Y'=Y/_{\{p_1=q_1\}},
$  and let $M'\in \Pic Y'$ be the (unique) line bundle corresponding to $M$.
As we just said,  Lemma~\ref{d1l} yields
$$
h^0(Y',M')=h^0(Y, M)-1.
$$
Suppose    $p_2\not\sim_{M_1}p_1$,
Then there is  $s_1\in H^0(Z_1,M_1)$   vanishing at $p_1$ but not at $p_2$. 
Hence $p_2$ is not a base point of $M_1(-p_1)$. By Lemma~\ref{bpl} we have
$ 
p_2\not\sim_{M'}q_2,
$, 
hence by Lemma~\ref{d1l}, for every $L\in F_{M'}(X)$ we have  $h^0(X,L)\leq h^0(Y',M')-1=h^0(Y, M)-2$. 

Conversely, assume $p_2\sim_{M_1}p_1$ and $q_2\sim_{M_2}q_1$.
 We claim that $p_2\sim _{M'}q_2$. 
Indeed, pick $s\in H^0(Y',M')$ such that $s(p_2)=0$. Call $s_i$ the restriction of $s$ to $Z_i$. Then
$s_1\in H^0(Z_1,M_1)$, hence $s_1(p_1)=0$ by hypothesis. Therefore $s_2(q_1)=0$. Finally, as 
$q_2\sim_{M_2}q_1$, we get $s_2(q_2)=0$, hence $s(q_2)=0$. So  $p_2\sim _{M'}q_2$.

By Lemma~\ref{d1l} this implies that there  exists $L\in F_{M'}(X)$
 such that $h^0(X,L)=h^0(Y',M')=h^0(Y, M)-1$, so we are
done.

If $\dd \geq 3$, we just apply the previous argument by replacing $p_2, q_2$ with $p_i, q_i$,\   $i\geq 3$,
and use  Remark~\ref{npr} (\ref{npe}).
\end{proof}
\begin{fact}
\label{cl00}
Let $X$ be   connected, and
assume  $\md =\mo=(0,\ldots, 0)$. Then for every
$L\in \picX{\mo}$ we have $h^0(X,L)\leq 1$ and equality holds  if and only if $L=\O_X$
(Corollary 2.2.5 of \cite{Ctheta}).
\end{fact}
The following easy observation will be applied several times.
\begin{remark}
\label{ur}
{\it Let $X=V\cup Z$ and $L\in \picX{\md}$; assume that $\md_Z=(0,\ldots, 0)$. Then
$h^0(X,L)\leq h^0(V, L_V)$.}

Indeed, let $Z=Z_1\coprod \ldots \coprod Z_c$ be the connected component decomposition of $Z$.
Then, by Fact~\ref{cl00}, $h^0(Z_i, L_{Z_i})\leq 1$ and equality holds  if and only if $ L_{Z_i}= \O_{Z_i}$,
in which case $L_{Z_i}$ has no base point. 
Set $X_1=V\cup Z_1\subset X$; if $h^0(Z_1, L_{Z_1})=0$ then, obviously, $h^0(X_1,L_{X_1})\leq h^0(V, L_V)$.
If instead $L_{Z_1}=\O_{Z_1}$,
by Lemma~\ref{d1l}  applied to $X_1$ we obtain $h^0(X_1,L_{X_1})\leq h^0(V, L_V)+1-1=h^0(V, L_V)$.
Iterating, we are done.
\end{remark}
Recall the notational  conventions of  \ref{not}.
\begin{lemma}
\label{e}
Let $X=C\cup Z$ with $C$ irreducible, set $\delta_C=C\cdot Z$.
Let $L\in \Pic X $ be such that  $deg L_C=2g_C+e_C$ for some $e_C\geq 0$.
Then 
\begin{enumerate}[(i)]
\item
\label{emin}
$
h^0(X,L)\leq h^0(C,L_C)+h^0(Z,L_Z)-\min\{\delta_C, e_C+1\}.
$
\item
\label{e=} If $e_C\geq \delta_C -1$ then $h^0(X,L)= h^0(C,L_C)+h^0(Z,L_Z)-\delta_C$.
\item
\label{e!}
If  $e_C\leq \delta_C -2$, equality holds in (\ref{emin}) for at most one $L$.
\end{enumerate}
\end{lemma}
\begin{proof}
We simplify the notation setting $\dd=\delta_C$.
Let $X_0:=C\coprod Z$ and $\nu_0:X_0\to X$ be the natural map (the normalization of $X$ at $C\cap Z$).
Call $M_0=(L_C,L_Z)\in \Pic X_0=\Pic C\times \Pic Z$.
We can factor $\nu_0$ by normalizing one node in $C\cap Z$ at the time, as follows. Denote
$$
\nu_0:X_0\stackrel{\nu ^0 _1}{\la} X_1 \stackrel{\nu ^1 _2}{\la} \ldots \la X_{\dd -1} 
\stackrel{\nu ^{\dd-1}
_\dd}{\la} X_{\dd}=X,
$$
so that 
$$
\nu ^i _{i+1}:X_i\la X_i /_{\{p_i=q_i\}}=X_{i+1}
$$ 
is the normalization of exactly one node of $X_{i+1}$, whose branches  $p_i, q_i$
satisfy $p_i\in C$ and $q_i\in Z$.  For all
$i<\dd$, denote
$\nu_i:X_i\la X$ the composition, and $M_i:=\nu_i^*L$. 
We have, of course,
\begin{equation}
\label{ieq}  h^0(X,L)\leq h^0(X_i,M_i).
\end{equation}
Notice that $h^0(X_0,M_0)= h^0(C,L_C)+h^0(Z,L_Z)$.

We claim that, for every $e\leq \min\{\dd -1, e_C\}$, we have
\begin{equation}
\label{eeq}  
h^0(X_{e+1},M_{e+1})= h^0(C,L_C)+h^0(Z,L_Z)-e-1.
\end{equation} 
By induction on $e$. 
If $e=0$, then $\deg L_C\geq 2g_C$, therefore $L_C$ has no base points.
By Lemma~\ref{d1l} we obtain  
$$
h^0(X_1,M_1)= h^0(X_0,M_0)-1=h^0(C,L_C)+h^0(Z,L_Z)-1.
$$
Assume, as induction hypothesis, that 
$h^0(X_{e},M_{e })= h^0(C,L_C)+h^0(Z,L_Z)-e.$
Now $$
\deg L_C(-\sum_{j=1}^ep_j)=\deg L_C - e \geq 2g_C,
$$
therefore $L_C(-\sum_{j=1}^ep_j)$ does not have base points;
in particular, $p_{e+1}$ is not a base point.
By Lemma~\ref{bpl} we have
$ 
p_{e+1}\not\sim_{M_e}q_{e+1}.
$ 
By Lemma~\ref{d1l}, this implies   
$$h^0(X_{e+1},M_{e+1})= h^0(X_{e },M_{e })-1=
h^0(C,L_C)+h^0(Z,L_Z)-e-1$$
proving (\ref{eeq}), 
 which, combined with (\ref{ieq}), proves (\ref{emin}). 

From  (\ref{eeq}) we also immediately derive 
(\ref{e=}).

Finally, for (\ref{e!}) it suffices to apply the uniqueness part of Lemma~\ref{d1l}.
\end{proof}

\subsection{Clifford index of a line bundle}
\label{index}
The   Clifford index  of a line bundle on a   curve $X$ is the number $\Cl L:=\deg L - 2h^0(X,L)+2$.
If $X$ is   irreducible 
and $0\leq \deg L \leq 2g$,  then $\Cl  L\geq 0$, by Clifford's theorem;
indeed
the extension to irreducible nodal of the classical Clifford's theorem for smooth curve is well known (and easy to prove by induction on the genus).
Notice also that  if $\Cl L=0$ then $L$ has no base points, and
if  $\Cl L=1$ then $L$ has at most one base point.
	
The next Lemma relates $\Cl  L$ to the equivalence $\sim_L$ defined in Definition~\ref{np}.
\begin{lemma}
\label{c} Let $C$ be an irreducible curve of genus $g$; fix $L\in \Pic ^dC$ with $h^0(L)\geq 2$ and 
$d\leq 2g$.
Let $E$ be a set of nonsingular points of $C$ such that $p\sim_L q$  for all $p,q\in E$. 
Then $\#E\leq \Cl L +2$.
\end{lemma}
\begin{proof}  Let $p_1,\ldots, p_e\in E$; for every $i=1,\ldots, e$  we have
$$
1\leq h^0(C,L-p_i)=h^0(C,L-\sum _{j=1}^ep_j)\leq \frac{d-e}{2}+1
$$ 
(by Clifford's theorem).  On the other hand  $h^0(C,L)=d/2+1-\Cl L/2$,
hence 
$$
h^0(C,L-p_i)\geq \frac{d  -\Cl L}{2}.
$$
Therefore
$$
  \frac{\Cl L-d}{2}\geq  \frac{e-d}{2}-1\  \Rightarrow \   \Cl L +2\geq e .
$$
\end{proof}
\begin{cor}
\label{cd}
Let 
$X=(C_1\coprod C_2)/_{\{p_i=q_i,\  \  i=1,\ldots,  \dd \}},$ with
$C_1$ and $C_2$   irreducible, and 
  $p_1,\ldots , p_{\dd}$ (resp. 
  $q_1,\ldots , q_{\dd}$)  nonsingular points of $C_1$ (resp. of $C_2$). 
Pick  $L_1\in \Pic C_1$  globally generated,    such
  that $h^0(C_1,L_1)\geq 2$ 
and $\Cl L_1+2< \dd$.
Then  for any $L_2\in \Pic C_2$ and any $L\in \FL$
we have $h^0(X,L)\leq h^0(C_1,L_1)+h^0(C_2,L_2) -2$.\end{cor}
\begin{proof}
Since $\dd >\Cl L_1+2$, Lemma~\ref{c} yields that there exists at least a pair
$p_i, p_j$ such that $p_i\not\sim_{L_1}p_j$.
As $L_1$ is globally generated, by Remark~\ref{npr}(\ref{npb}) we have
$p_i\not\sim_{L}q_i$ for any $L$ as above; hence Lemma~\ref{d2} applies,
  giving the statement.
\end{proof}
In what follows we shall frequently use, without mentioning it, the obvious fact that $\Cl L$ and $\deg L$ have the same parity.
\begin{prop}
\label{ic} Let $X=C_1\cup C_2$ with 
$C_i$   irreducible of genus $g_i$. 
Assume $\dd:=C_1\cdot C_2\geq 2$. Let $L\in \picX{\md}$,
set $L_i=L_{C_i}$, $d_i=\deg _{C_i}L_i$  and assume $0\leq d_i\leq 2g_i$
for $i=1,2$.
\begin{enumerate}[(i)]
\item

\label{ic0}
If $\Cl L =0$ then 
$\Cl  L_1=\Cl L_2=0$; moreover, if $\md \neq \mo$ then  $\dd= 2$.
\item
\label{ic1}
If $\Cl L =1$  we may assume $d_1$ odd and $d_2$ even. Then
$\Cl  L_1= 1$ and $\Cl L_2=0$. Moreover, if $d_1\geq 3$, then $\dd\leq 3$;
if $d_2\geq 2$ then $\dd =2$.
\item
\label{ic01} If  $0\leq \Cl L \leq 1$, then
$$
h^0(X,L)\leq h^0(C_1,L_1)+h^0(C_2,L_2)-1\leq d/2+1.
$$
\end{enumerate}
\end{prop}
\begin{proof} Denote $l=h^0(X,L)$ and $l_i=h^0(C_i, L_i)$.
Let $p_1,\ldots , p_{\dd}\in
C_1$  and $q_1,\ldots , q_{\dd}\in C_2$ be the points corresponding to the nodes of
$X$, so that
$$X=(C_1\coprod C_2)/_{\{p_i=q_i,\  \  i=1,\ldots,  \dd \}}.$$ 
Now, as $l\leq l_1+l_2$ we always have
\begin{equation}
\label{yb}
\Cl L=d-2l+2\geq d-2l_1-2l_2+2=\Cl L_1+ \Cl L_2-2.
\end{equation}
Moreover, if either $L_1$ does not have a base point at some $p_i$, or $L_2$ does not have a base point at some $q_i$,
we have $l\leq l_1+l_2-1$, by Lemma~\ref{d1l} . Therefore
\begin{equation}
\label{nb}
\Cl L=d-2l+2\geq d-2l_1-2l_2+2+2=\Cl L_1+ \Cl L_2.
\end{equation}
Recall that if  $\Cl L_i\leq 1$ then $L_i$ has at most one base point. Therefore,
as $\dd \geq 2$,   (\ref{nb}) applies  if either $\Cl L_1\leq 1$ or $\Cl L_2\leq 1$.

Assume $\Cl L=0$.
Then (\ref{yb}) yields $\Cl L_i\leq 2$ for $i=1,2$ (as $\Cl L_i\geq 0$ by Clifford's theorem for irreducible curves).
 If $\Cl L_1= 0$ we can apply (\ref{nb}), obtaining  $\Cl L_2=0$.
Moreover  we have equality occurring in (\ref{nb}), hence $l=l_1+l_2-1$.
By Lemma~\ref{d2} we obtain that $p_i\sim _{L_1}p_j$ and $q_i\sim _{L_2}q_j$ for all $i, j$.
If $\md \neq \mo$ and $\dd \geq 3$, this is impossible by  Lemma~\ref{c}. We conclude that $\dd=2$.

By switching roles between $L_1$ and $L_2$ this argument together with  (\ref{yb}) 
 shows that $\Cl L=0$ implies $\Cl L_i\leq 1$ for $i=1,2$.
If $\Cl L_1= 1$ applying (\ref{nb}) gives $0\geq 1+\Cl L_2$, which is impossible. (\ref{ic0}) is proved.

Now assume $\Cl L=1$;
(\ref{yb}) yields $\Cl L_1+\Cl L_2 \leq 3$.   
 If $\Cl L_1= 1$ then (\ref{nb}) applies; we get
$1\geq 1+\Cl L_2$, hence $\Cl L_2=0$. Similarly, if $\Cl L_2=0$ by (\ref{nb}) we get $\Cl L_1=1$.
We thus have that $\Cl L_1= 1$ if and only if $\Cl L_2=0$.
As $d_1$ is odd, the only remaining case is $\Cl L_1=3$; this would imply $\Cl L_2=0$
which implies $\Cl L_1= 1$, a contradiction. Therefore the case $\Cl L_1=3$ does not occur.
In a similar way we see that the case $\Cl L_2=2$ cannot occur (it would imply  $\Cl L_1= 1$
which implies $\Cl L_2=0$).

Finally,  equality holds in (\ref{nb}), so that $l=l_1+l_2-1$.
Hence $p_i\sim _{L_1}p_j$ and $q_i\sim _{L_2}q_j$ for all $i, j$ (by Lemma~\ref{d2} as before).
Now, if either $d_1\geq 3$ and $\dd \geq 4$,
or if  $d_2\geq 2$ and $\dd \geq 3$, this is impossible by  Lemma~\ref{c}.
   (\ref{ic1}) is proved.

Part (\ref{ic01}) follows from the previous ones, observing that in both cases $L_2$ has no base
points. Therefore by Lemma~\ref{d1l}  we have $l\leq l_1+l_2-1$.
Finally, if $\Cl L=0$ we have $l_1+l_2-1=d_1/2+1+d_2/2+1-1=d/2+1$.
If $\Cl L=1$ we have $l_1+l_2-1=(d_1+1)/2 +d_2/2+1-1=d/2+1/2$; so we are done.
\end{proof}

\section{Riemann's  theorem for  semistable curves}
\label{Riemann}
The well known Riemann's theorem for a smooth curve $C$ of genus $g$
states the following: if $d\geq 2g-1$ and $L\in \Pic^dC$, then $h^0(C,L)=d-g+1$.
More generally, using  the normalization and induction on the number of nodes,  it is easy to prove the following:
\begin{fact}
\label{bpfirr} Let $X$ be a nodal irreducible curve (of genus $g$) and $L\in \Pic^dX$. Then
\begin{enumerate}[(1)]
\item
\label{irr1}
If $d\geq 2g-1$ then $h^0(X,L)=d-g+1$.
\item
\label{irr2}
If $d\geq 2g$ then $L$ is free from base points.
\end{enumerate}
(Part (\ref{irr1}) follows from Riemann-Roch and Serre duality, 
  (\ref{irr2}) follows from   (\ref{irr1})).
\end{fact}

 By contrast, if $X$ is   reducible, Riemann's theorem trivially fails.
In fact, for every fixed $d\geq 2g-1$ 
there exist  infinitely many  multidegrees $\md$, with $|\md|=d$,
such that for any $L\in \picX{\md}$ we have
$ 
h^0(X,L)>d-g+1
$  (see Example~\ref{Pr}).

On the other hand, it is well known that, for every $d$, there exists a well defined finite set of multidegrees,
of  total degree $d$, which appear as the multidegrees of all line bundles parametrized by the compactified Picard variety
of a stable curve $X$.  
 More precisely, for any stable curve $X$ 
 we shall denote by $\PXb$ the compactified Picard scheme constructed (independently)
 in
 \cite{OS}, \cite{simpson}, \cite{caporaso}, \cite{pandha} (known to be all isomorphic
 by \cite{alex} and \cite{pandha}).
 Recall that $\PXb$ is a reduced scheme of pure dimension $g$, which appears as the specialization
 of   the degree-$d$ Picard varieties of smooth curves specializing to $X$.
 There are several modular descriptions of $\PXb$; the one we shall use interprets its points as
 equivalence classes of balanced line bundles on curves stably equivalent to $X$.  

The main result of this section, Theorem~\ref{sr}, states that if $L$ is a line bundle on a semistable curve $X$,
having degree at least $2g-1$, and balanced multidegree, then, just as for smooth curves, we have
$h^0(X,L)=d-g+1.$
Therefore,  if $X$ is stable, every line bundle  parametrized by the compactified   Picard scheme 
$ \PXb$ satisfies Riemann's theorem.

\subsection{Balanced   line bundles.}
\label{w}
   Let $X$ be fixed. For every subcurve $Z\subset X$
with $\delta_Z:= Z\cdot Z^c$, we set 
\begin{equation}
\label{wZ}
w_Z:=\deg_Z \omega _X=2g_Z-2+\delta_Z\   \  \  \text{and }\  \  w:=w_X=2g-2.
\end{equation}
Recall that
a (nodal connected) curve $X$ of genus $g\geq 2$ is {\it stable} if  for every subcurve
$Z\subset X$ we have
$
0<w_Z<w.
$
$X$ is   {\it semistable} if
for every  $Z\subset X$ we have
\begin{equation}
\label{DMss}
0\leq w_Z\leq w,
\end{equation}
and $w_Z=0$  if and only if $Z$ is a union of exceptional components of $X$ 
(a component $E\subset X$ is called exceptional if $E\cong \pr{1}$ and if $\dd_E=2$).

We say that a semistable curve $X$ is {\it stably equivalent} to a stable curve $\ov{X}$ if
$\ov{X}$ is the curve obtained from $X$ by contracting all of its exceptional components.
$\ov{X}$ is called the {\it stabilization} of $X$.

\begin{nota} 
\label{bal} Let $\md \in \Z^{\gamma}$ with $d=|\md|$; also fix $g\geq 2$. 
Assume that $X$ is  stable. We say that $\md$ is {\it balanced}  if for every
(connected) subcurve $Z\subset X$ we have
\begin{equation}
\label{BI}
d\frac{w_Z}{w}-\frac{\delta_Z}{2}\leq d_Z \leq d\frac{w_Z}{w}+\frac{\delta_Z}{2}.
\end{equation}


More generally, if $X$ is semistable, we say that $\md$ is balanced  if (\ref{BI}) holds, 
and if for every
exceptional component $E$ of  $X$ we have
$d_E=1$ (note that if a semistable curve admits some balanced multidegree, then it is quasistable, i.e. two exceptioanl components do not intersect).
Set
\begin{equation}
\label{BIs}
\BXd:=\{\md: |\md |=d, \  \md \text{ is balanced}\}.
\end{equation}
 A line bundle on a semistable curve  is balanced if its multidegree is balanced.

\begin{example}
\label{ct2}
Let $X=C_1\cup C_2$ with $C_1\cdot C_2=1$ and $1\leq g_1 \leq g_2$.
Pick $d=2$.  
\begin{displaymath}
B_2(X)= \left\{ \begin{array}{l}
\{(0,2)\}  \  \  \  \  \ \  \  \  \text{ if }   g_1< (g+1)/4\\
 \{(0,2);(1,1)\}  \text{ if }   g_1= (g+1)/4\\
\{(1,1)\}   \  \  \  \  \ \  \  \ \text{ if }   g_1> (g+1)/4\\
\end{array}\right.
\end{displaymath}
\end{example}

The terminology ``balanced" 
was introduced in \cite{caporaso}  to indicate that balanced multidegrees
are closely related to the topological characters of the curve.
Indeed, the balanced multidegrees of total degree $d\in \Z$
are as close as they can be to the multidegree $\frac{d}{2g-2}\mdeg \omega_X$.
The word balanced is
 sometimes replaced by the word ``semistable".
As we mentioned at the beginning of the section, if $X$ is stable its compactified Picard scheme parametrizes
equivalence classes of balanced line bundles on semistable curves having $X$ as stabilization. If $X$ is semistable, its compactified Picard scheme turns out to coincide to the compactified Picard scheme of  its stabilization. In the present paper we do not need to be more precise about this point; see loc. cit for details.
\end{nota}

\subsection{Positivity properties of balanced line bundles.}
We denote  
\begin{equation}
\label{sep}
\sep:=\{n\in \sing: n  \text{ is a separating node of } X\}\subset X.
\end{equation}

\begin{thm}[Balanced Riemann]
\label{sr} Let $X$ be a semistable curve of genus $g\geq 2$,
$d$ an integer 
and $\md \in \BXd$. Let $L\in \picX{\md}$.
\begin{enumerate}[(i)]
\item
\label{sr1}
If $d\geq 2g-1$, then $ 
h^0(X,L)=d-g+1.
$
\item
\label{sr2}
If $d\geq 2g$ and $\sep = \emptyset$, 
then     $L $ has no base points.
\item
\label{sr3}
If $d\geq  5(g-1) $, 
then     $L $ has no base points.
\end{enumerate}
\end{thm}

Part (\ref{sr1}) may fail if  $\md$ is not balanced; see Example~\ref{Pr}.
Part (\ref{sr2})  may fail if   $\sep \neq \emptyset$; see Example~\ref{Pr2}.
 
\begin{proof}
Let  $Z\subsetneq X$  be a connected subcurve.
We claim that, if $d\geq 2g-1$, we have
\begin{equation}
\label{dZ1}
d_Z\geq 2g_Z -1
\end{equation}
and, if $d\geq 2g$ and $\sep =\emptyset$, we have
\begin{equation}
\label{dZ2}
d_Z\geq 2g_Z.
\end{equation}
To prove this, set $d=2g-2+a=w+a$ with $a>  0$.
As $\md$ is balanced, we have
$$
d_Z\geq d\frac{w_Z}{w}-\dfZ=2g_Z-2+\dfZ+a\frac{w_Z}{w}.
$$
Now, $\delta_Z \geq 1$ and  $w_Z\geq 0$ (cf. (\ref{DMss})). Therefore
the above inequality yields
$d_Z\geq 2g_Z -1 $, as claimed in (\ref{dZ1}).

To prove  (\ref{dZ2}),   assume $\sep =\emptyset$. Then
 $\dd_Z\geq 2$, so
the previous
inequality yields
$d_Z\geq 2g_Z$, unless $w_Z=0$, i.e. unless 
$Z$ is a chain of exceptional components
(recall that $X$ is semistable). If that is the case,
 $d_Z=1$ and $g_Z=0$.  So we  have $d_Z=2g_Z+1>2g_Z$. 
(\ref{dZ2}) is proved.

Now, part (\ref{sr1}) of the Theorem follows from   the next Lemma~\ref{er}.

We shall apply Lemma~\ref{er} also for part (\ref{sr2}). If $d_Z\geq 2g_Z$ for every $Z$,
then for any nonsingular point $p\in X$ we obviously have $\deg_ZL(-p)\geq 2g_Z-1$,
hence Lemma~\ref{er} applies to $L(-p)$, yielding $h^0(X,L(-p))=h^0(X,L)-1$.
Now let $n\in \sing$. Let $\nu:Y\to X$ be the normalization of $X$ at $n$,
$M:=\nu^*L$ and $\nu^{-1}(n)=\{q_1,q_2\}$.
To prove that $L$ has a section not vanishing at $n$ it suffices to prove that 
\begin{equation}
\label{Mq12}
h^0(Y,M(-q_1-q_2))=h^0(Y,M)-2.
\end{equation}
Let $Z'\subset Y$ be a connected subcurve, and $Z:=\nu(Z')$. Then
$$
\deg_{Z'}M=\deg_ZL\geq 2g_Z,
$$  also $g_{Z}\geq g_{Z'}$ and strict inequality holds if and only if both $q_1$ and $q_2$ lie on $Z'$,
in which case $g_{Z}=g_{Z'}+1$.
Therefore
\begin{displaymath}
deg  _{Z'}M(-q_1-q_2)\geq \left\{ \begin{array}{l}
2g_{Z}-2=2g_{Z'},  \  \text{ if }  q_1,q_2\in Z'\\
\\
2g_{Z}-1\geq 2g_{Z'} -1,  \text{ otherwise. }\\
\end{array}\right.
\end{displaymath}
We can thus apply Lemma~\ref{er},  proving (\ref{Mq12}) as follows:
$$
h^0(Y,M(-q_1-q_2))=\deg M-2-g_Y+1=h^0(Y, M)-2.
$$ 

By the same argument, to prove  (\ref{sr3}) it suffices to show that $d_Z\geq 2g_Z$ for every $Z\subset X$.
Now,  $d\geq 5(g-1)$ implies $d\geq 2g$, so by
the previous parts it suffices to consider subcurves $Z$ having $\dd_Z=1$.
Let $Z$ be such a subcurve of $X$; note that $g_Z\geq 1$ ($X$ is semistable) hence 
$w_Z=2g_Z-2+\dd_Z\geq 2-2+1=
1$. As $\md$ is balanced, and $d\geq  2(g-1)+3(g-1)=w+3(g-1)$, we have
$$
d_Z\geq \frac{dw_Z}{w}-\frac{1}{2}\geq w_Z+\frac{3(g-1)w_Z}{2(g-1)}-\frac{1}{2}= 2g_Z-\frac{3}{2}+\frac{3w_Z}{2}\geq 2g_Z.
$$
Hence  we are done.
 \end{proof}
\begin{lemma}
\label{er}
Let $Y$ be a  (possibly disconnected) curve of genus $g$ and $L\in \Pic^d Y$.
If $\deg_ZL\geq 2g_Z-1$  for every connected subcurve $Z\subseteq Y$,
 then  
$ 
h^0(Y,L)=d-g+1.
$ 
\end{lemma}
\begin{proof}
Let $X_1,\ldots, X_c$  be the  connected components of $Y$. Then $g=\sum_{i=1}^cg_{X_i}-c+1$ and
$h^0(Y,L)=\sum_{i=1}^ch^0(X_i, L_{X_i})$; therefore it suffices to prove the lemma for a connected curve $X$ of genus
$g$.

We shall use induction on the number of irreducible components of $X$. The base case, $X$ irreducible,
is known (cf. Fact~\ref{bpfirr}). Assume $X$ reducible.
We begin by showing that there exists an irreducible component, $C_1$, 
  of $X$ such that 
\begin{equation}
\label{d1}
d_1\geq 2g_1+\delta_1-1.
\end{equation}
By contradiction, assume the contrary. Then
$$
d=\sum_{i=1}^{\gamma}d_i\leq \sum_{i=1}^{\gamma}(2g_i+\delta_i-2)=2\sum_{i=1}^{\gamma}g_i
+\sum_{i=1}^{\gamma}\delta_i-2\g.
$$
Now, $\sum_{i=1}^{\gamma}\delta_i=2\delta$  and $g=\sum_{i=1}^{\gamma} g_i+\delta -\g +1$.
Therefore
$$
d\leq 2(\sum_{i=1}^{\gamma}g_i+
 \delta - \g)=2(g-1),
$$
contradicting the assumption $d\geq 2g-1$. This proves (\ref{d1}).

Let us write $X=C_1\cup Z$ with $Z=  C_1^c$. Let $Z=Z_1\coprod\ldots\coprod Z_c$, with  
$Z_i$ 
  connected. We  use induction  and get
\begin{equation}
\label{Zi}
h^0(Z_i,L_{Z_i})=d_{Z_i}-g_{Z_i}+1.
\end{equation}
Now, by (\ref{d1}) we can apply Lemma~\ref{e}(\ref{e=}) and obtain
$$
h^0(X,L)= h^0(C_1,L_1)+h^0(Z,L_Z)-\delta_1=d-(g_1+\sum_{i=1}^cg_{Z_i})+c+1-\delta_1
$$
(using $h^0(C_1,L_1)=d_1-g_1+1$ and (\ref{Zi})). Now $g=g_1+\sum_{i=1}^cg_{Z_i}+\delta_1-c$,
hence
$ 
h^0(X,L)= d-g+\delta_1-c+c+1-\delta_1=d-g+1.
$ 
\end{proof}

\begin{example}
\label{Pr}
Fix  $X$ having $\g\geq 2$ components  and genus $g$;   let $d\geq 2g-1$. The theorem of Riemann
fails for all but finitely many
$\md$ with $|\md|=d$. To prove that it will be enough to show
 the following. For every fixed $i\in \{1,\ldots, \g\}$ there exists $m_i$ such that
for every
$\md$ such that $d_i\geq m_i$ and for every $L\in \picX{\md}$ we have $h^0(X,L)>d-g+1$.

So, pick $i=1$, let $m_1:=d+g_1+\dd_1+1$
($\dd_1=C_1\cdot C_1^c$). If $d_1\geq m_1$ we have
$$
d_1\geq d+g_1+\dd_1+1 \geq 2g-1 +g_1+\dd_1+1\geq 2g_1  +g_1+\dd_1 =3g_1+\dd_1\geq 2g_1+1;
$$
hence $h^0(C_1,L_1)=d_1-g_1+1$.
Now, for any $L\in \picX{\md}$ such that $d_1\geq m_1$
(we can adjust the remaining $d_2,\ldots, d_{\g}$ however we like
so that $|\md|=d$)
$$
h^0(X,L)\geq h^0(C_1,L_1)-\dd_1=d_1-g_1+1-\dd_1\geq d+g_1+\dd_1+1-g_1+1-\dd_1
$$
hence $h^0(X,L)\geq d+2>d-g+1$ as wanted.
\end{example}

\begin{example}
\label{Pr2} If $X$ has a separating node  part (\ref{sr2}) of Theorem~\ref{sr}  may fail.
Let $X=C_1\cup C_2$ with $C_1\cdot C_2=1$. Assume $g_1=1$ and $g_2=g-1$ and $d=2g+b$ with $b\geq 0$.
Let $\md=(1,d-1)=(1,2g+b-1)=(1,2g_2+b+1)$, if $g\geq b+3$ one checks that $\md$ is balanced. 
Set $L=(\O_{C_1}(p),L_2)$ such that $p\neq C_1\cap C_2$.
Assume for simplicity that $L_2$ has no base point in $C_1\cap C_2$. Then
$$
h^0(X,L)=h^0(C_1,\O_{C_1}(p))+h^0(C_2,L_2)-1=h^0(C_2,L_2).
$$
Now,  $L$ has a base point in $p$, indeed
$$
h^0(X,L(-p))=h^0(C_1,\O_{C_1} )+h^0(C_2,L_2)-1=h^0(C_2,L_2).
$$
\end{example}

\section{Clifford's Theorem for all degrees}
In this section we prove the following cases of Clifford's theorem: 
Theorem~\ref{clvine}, for curves with two components and every 
balanced multudegree;  Theorem~\ref{cc} for all curves and all balanced line bundles of  degree $2g-2$;  Proposition~\ref{ecl}
for all curves and all degrees, provided the hypothesis that the degree be at most twice the genus is ``uniformly" satisfied  on all irreducible components.

\subsection{Uniform extension}

\begin{prop}[Uniform Clifford]
\label{ecl}
Let $X$ be a connected  curve of genus $g$.
Let $\md=(d_1,\ldots, d_{\g})\in \Z^{\gamma}$ be such that
$0\leq d_i\leq 2g_i$ for every $i=1,\ldots, \gamma$.
\begin{enumerate}[(i)]
\item
\label{ec}
Then $|\md|\leq 2g$ and 
for every $L\in \picX{\md}$ we have
$ h^0(X,L)\leq   \deg L/2 +1.
$ 
\item
\label{ec=}
If equality holds and $|\md|\leq 2g-2$   then 
$L$ has no nonsingular base points
(i.e. if $L$ admits a base point, this point is a node of    $X$).
\end{enumerate}
\end{prop}

\begin{proof}
As we said in Subsection~\ref{index} we may assume $X$ reducible. Set $|\md|=d$. 

Let us prove that    $d\leq 2g$.  
 We have
$ 
d=\sum_{i=1}^{\gamma}d_i\leq \sum_{i=1}^{\gamma} 2g_i .
$ 
Let $\delta$ be the number of nodes of $X$ that lie in two different irreducible components.
Then
$g=\sum_{i=1}^{\gamma}g_i+\delta -\g +1$. On the other hand, as $X$ is connected, we have
$ 
\delta\geq \g -1.
$ 
Therefore
$ 
2g-d\geq 2g-2\sum_{i=1}^{\gamma}g_i=2(\delta- \g+ 1)\geq  0,$ as claimed. 

We continue using induction on the number of irreducible
components.

By Remark~\ref{ord}, we can decompose $X=Z_1\cup Z_2$ so that the $Z_i$ are connected. We set $l_i:=h^0(Z_i, L_{Z_i})$;
 by  the induction assumption,
$
l_i\leq \frac{d_{Z_i}}{2}+1
$
and if equality holds,   $L_{Z_i}$ has no nonsingular base points.
We distinguish three cases.

\noindent
{\bf Case 1:} $l_i<\frac{d_{Z_i}}{2}+1$ for both $i=1,2$.

If
$d_{Z_1}$ and $d_{Z_2}$ are even, then $l_i\leq \frac{d_{Z_i}}{2}$. Hence $
h^0(X,L)\leq l_1+l_2\leq \frac{d}{2}.$
 
If  $d_{Z_1}$ is even and $d_{Z_2}$ is odd, then $l_1\leq \frac{d_{Z_1}}{2}$ and $l_2\leq \frac{d_{Z_2}+1} {2}$. Hence
$
h^0(X,L)\leq l_1+l_2\leq \frac{d+1}{2}< \frac{d}{2} +1.
$

Finally, assume  $d_{Z_1}$ and $d_{Z_2}$   odd. Then $l_i\leq \frac{d_{Z_i}+1}{2}$ hence 
$$
h^0(X,L)\leq l_1+l_2\leq   \frac{d}{2} +1.
$$
If equality holds   we get $l_i= \frac{d_{Z_i}+1}{2}$   for $i=1,2$, and 
$ 
h^0(X,L)= l_1+l_2.
$ 
Therefore $L_{Z_1}$and $L_{Z_2}$ have a base point over every node in $Z_1\cap Z_2$.
This implies that
$
Z_1\cdot Z_2=1.
$
Indeed, by induction, the Clifford inequality  holds on $Z_i$,
yielding that $L_{Z_i}$ can have at most one base point
(indeed, if $L_{Z_i}$ had two base points, $p$ and $p'$, then
$h^0(L_{Z_i}  (-p-p')) =h^0(L_{Z_i})=\frac{d_{Z_i}+1}{2}>\frac{d_{Z_i}-2}{2}+1$). 

Let $q_i\in Z_i$ be the branch of the node  $n=Z_1\cap Z_2$. 
Let $p\in X$ be a 
point with $p\neq n$, say $p\in Z_1$.
If $p$ is a base point for $L$ then it is also a base point for $L_{Z_1}$, but this is not possible as
we just proved that the only   base point of $L_{Z_1}$ is $q_1$.

The proof of (\ref{ec}) and (\ref{ec=})  in Case 1 is complete.

\noindent{\bf Case 2:} $l_1=\frac{d_{Z_1}}{2}+1$ and $l_2<\frac{d_{Z_2}}{2}+1$.

By induction, $L_{Z_1}$ has no  nonsingular base point. Therefore, by Lemma~\ref{d1l}  
$$
h^0(X,L)\leq l_1+l_2-1< \frac{d_{Z_1}}{2}+1  +  \frac{d_{Z_2}}{2}+1 -1=\frac{d}{2} +1.
$$
So, in this case strict inequality always holds and we are done.

\noindent{\bf Case 3:} $l_i=\frac{d_{Z_i}}{2}+1$ for both $i=1,2$.

By induction $ L_{Z_i}$ is  free from nonsingular base points. We get, again by Lemma~\ref{d1l},
$$
h^0(X,L)\leq l_1+l_2-1=  \frac{d_{Z_1}}{2}+1  +  \frac{d_{Z_2}}{2}+1 -1=\frac{d}{2} +1.
$$
Now equality holds  if and only if $h^0(X,L)= l_1+l_2-1$.
Let $p\in X$ be a nonsingular point, say $p\in Z_1$. As  $p$ is not a base point of $L_{Z_1}$, we have
$$
h^0(X,L(-p))\leq h^0(Z_1,L_{Z_1}(-p))+l_2-1= l_1-1+l_2-1=h^0(X,L)-1
$$
hence  $p$ is not a base point of $L$, so we are done.
\end{proof}

\begin{cor}
\label{ecor} Assumptions as in Proposition~\ref{ecl}. Assume  $0<|\md|<2g-2$.
If there exists 
$L\in \picX{\md}$ such that $\Cl  L=0$,
then for every decomposition
 $X=Z_1\cup Z_2$ with $Z_1$ connected and   $Z_2$ irreducible, we have
\begin{enumerate}[(a)]
\item
\label{} 
$ Z_1\cdot Z_2 \leq 2,$  
\item
\label{} 
If   $d_{Z_1}$ and $d_{Z_2}$  are even,   then  $\Cl  L_{Z_i} =0$ and
$h^0(Z_i, L_{Z_i}(-Z_1\cap Z_2))=h^0(Z_i, L_{Z_i})-1$, for $i=1,2 $.
\item
\label{}
If $d_{Z_1}$ and $d_{Z_2}$   are odd, then $Z_1\cdot Z_2  =1$   and  $\Cl  L_{Z_i}(-Z_1\cap Z_2) =0$ for $i=1,2$.
\end{enumerate}
\end{cor}
\begin{proof} We use the proof 
of Proposition~\ref{ecl}.  In Case 1, $\Cl L=0$  exactly when the
$d_{Z_i}$    are both odd,
$Z_1$ and $Z_2$ intersect in only one point, and 
 $$ h^0(Z_i,L_{Z_i})=h^0(Z_i,L_{Z_i}(-q_i))=\frac{d_{Z_i}+1}{2}=\frac{d_{Z_i}-1}{2}+1.$$
So $\Cl  (L_{Z_i}(-q_i)) =0$.
Observe that we did not   use the irreducibility of $Z_2$.

In Case 2 equality never holds. 

In Case 3 we  have $\Cl  L=0$ exactly when the $d_{Z_i}$   are even,
 $\Cl  L_{Z_i} =0$ for
$i=1,2$, and $h^0(X,L)=h^0(Z_1,L_{Z_1})+h^0(Z_2,L_{Z_2})-1$.
Notice that by Lemma~\ref{d2} this implies that for every pair of points $q,q'\in Z_1\cap Z_2\subset Z_2$ we have
$q\sim _{L_{Z_2}}q'$ (and similarly for $Z_1$).

To complete the proof, we need to show that $ Z_1\cdot Z_2 \leq 2$. By contradiction, assume $Z_1\cdot Z_2 \geq 3;$
then a relation  $q\sim _{L_{Z_2}} q'\sim _{L_{Z_2}}q ''$ holds on $Z_2$.
Observe also that $L_{Z_2}$ has no nonsingular base points, as $\Cl  L_{Z_2} =0$. Therefore  
$$
h^0(Z_2,L_{Z_2}(-q-q'-q''))=h^0(Z_2,L_{Z_2}(-q))=l_2-1=\frac{d_{Z_2}}{2}.
$$
But    $Z_2$ is irreducible, hence Clifford applies to $L_{Z_2}(-q-q'-q'')$, and we get
$$
h^0(Z_2,L_{Z_2}(-q-q'-q''))\leq \frac{d_{Z_2}-3}{2}+1<\frac{d_{Z_2}}{2},$$  
a contradiction.
\end{proof}

\subsection{Curves with two components}
Clifford's inequality holds for curves with two irreducible components, by the following result.
\begin{thm}
\label{clvine}
Let $X=C_1\cup C_2$ be a semistable curve of genus $g\geq 2$.
Let $0\leq d\leq 2g$ and $\md \in \BXd$.  Then
for every $L\in \picX{\md}$
 we have \begin{equation}
\label{cleq}
h^0(X,L)\leq d/2+1.
\end{equation}

 \end{thm}
\begin{add}
\label{add2} Let $\epsilon:=1+\max\{d_1-2g_1,d_2-2g_2,0\}$, and $\beta:=\min\{C_1\cdot C_2,\epsilon\}$. 
If $C_1\cdot C_2\geq 2$, then
$
h^0(X,L)\leq h^0(C_1,L_1)+h^0(C_2,L_2)-\beta\leq d/2+1.
$
\end{add}

\begin{proof}
Set $l:=h^0(X,L)$, and for $i=1,2$,\   $L_i:=L_{C_i}$,\  $l_i:=h^0(C_i,L_i)$.
As usual, set $\dd:=C_1\cdot C_2$.
By Theorem~\ref{sr} we can  assume $d\leq 2g-2$.
We begin with

\noindent
{\bf Case 0.} {\it If $d_1<0$ then (\ref{cleq}) holds, with strict inequality if $d\leq 2g-2$.}

As $d_1<0$ we have  $d_2>0$. Since $\md$ is balanced, 
\begin{equation}
\label{eqd1}
d_1\geq \frac{dw_1}{w}-\df\geq  -\df
\end{equation}
($\frac{dw_1}{w}\geq 0$ as $X$ is semistable).
Of course $l_1=0$, therefore, denoting by $G_2\in \Div C_2$ the degree $\dd$ divisor cut
on $C_2$ by $C_1$, a section of $L$ has to vanish on $G_2$, i.e. 
\begin{equation}
\label{eq2}
h^0(X,L)=h^0(C_2,L_2(-G_2)).
\end{equation}
Note that $\deg L_2(-G_2)=d_2-\dd$.
If $d_2-\dd< 0$ we get $h^0(X,L)= 0$ and we are done.
If  $0\leq d_2-\dd\leq 2g_2$ we can use Clifford on $C_2$ and obtain
$$
h^0(C_2,L_2(-G_2))\leq \frac{d_2-\dd}{2}+1=\frac{d-d_1-\dd}{2}+1\leq  \frac{d+\dd/2-\dd}{2}+1
$$
(using (\ref{eqd1})). Combining the above with (\ref{eq2}) yields
$$
h^0(X,L)\leq \frac{d }{2}+1-\frac{\dd }{4}<\frac{d }{2}+1
$$ 
as stated.
Finally, it remains to treat the case $d_2-\dd\geq 2g_2$, i.e.
$$
l=h^0(C_2,L_2(-G_2))= d_2-\dd-g_2+1.
$$
We argue by contradiction,   assuming that $l\geq \frac{d}{2}+1$.
This is to say, by (\ref{eq2}),
$$
d_2-\dd-g_2+1\geq \frac{d}{2}+1,
$$
hence (using $d=d_1+d_2$)
$$
\frac{d_2-d_1}{2}-\dd-g_2\geq 0,
$$
equivalently
\begin{equation}
\label{i1}
d_2-d_1-2\dd-2g_2\geq 0.
\end{equation}
On the other hand, as $\md$ is balanced, we have
$$
d_2\leq\frac{dw_2}{w}+\df \  \  \  \text{ and } \  \  \    d_1\geq\frac{dw_1}{w}-\df.
$$
Using these two inequalities we get
$$
d_2-d_1-2\dd-2g_2\leq \frac{dw_2}{w}+\df- \frac{dw_1}{w}+\df -2\dd-2g_2=\frac{d}{w}(w_2-w_1)-\dd-2g_2.
$$
Now, $w_2-w_1=2g_2-2g_1$ and $\frac{d}{w}\leq 1$ (as $d\leq 2g-2=w$). We obtain
$$
d_2-d_1-2\dd-2g_2\leq\frac{d}{w}(2g_2-2g_1)-\dd-2g_2\leq - \frac{2dg_1}{w}-\dd<0
$$
contradicting (\ref{i1}). This finishes Case 0.

For the rest of the proof, we can restrict to $d_i\geq 0$ for $i=1,2$.
By  Propositions~\ref{ecl} and \ref{ic} (\ref{ic01}),
we can assume that $d_i\geq 2g_i+1$ for at least one $i$, so let $d_1\geq 2g_1+1$.
Then 
 $l_1=d_1-g_1+1.$

\noindent {\bf Case 1.}  {\it If $d_1\geq 2g_1+\delta -1$, then  (\ref{cleq}) holds, with strict inequality 
if $d\leq 2g-1$.}

By Lemma~\ref{e}(\ref{e=}),
\begin{equation}
\label{vd}
l=l_1+l_2-\delta.
\end{equation}
\noindent
{\it Subcase 1a.} 
$d_2\geq 2g_2$. Hence 
$ 
l_2 = d_2-g_2+1.
$ 
Combining with (\ref{vd}) we have
$$
l=d_1-g_1+1+d_2-g_2+1-\delta=d-(g_1+g_2+\dd -1)+1= d-g+1.
$$
Now $d\leq 2g$, hence
$$
l= d-g+1\leq d-\frac{d}{2}+1=\frac{d}{2}+1.
$$
So we are done.
Note that equality holds if and only if $d=2g$.
 
\noindent
{\it Subcase 1b.} $d_2< 2g_2$. By Proposition~\ref{ecl},  $l_2\leq \frac{d_2}{2}+1$.
Set
$$
d_1=2g_1+\delta -1+a 
$$
so that $a\geq 0$ and 
\begin{equation}
\label{vg}
g_1=\frac{d_1-\dd+1-a}{2}.
\end{equation}
Using  (\ref{vd}) and (\ref{vg}) we get
$$
l\leq d_1-g_1+1+ \frac{d_2}{2}+1-\dd= d_1-\frac{d_1-\dd+1-a}{2}+2+ \frac{d_2}{2}-\dd,
$$
hence
$$
l\leq   \frac{d}{2}+1+\frac{1-\dd+a}{2}.
$$
The subsequent Lemma~\ref{ae} yields
\begin{displaymath}
a \leq \left\{ \begin{array}{l}
\frac{\delta}{2}-1, \  \text{ if }  \  \dd  \text{ is even}\\
\\
\frac{\delta -1}{2}-1, \text{ if }  \ \dd  \text{ is odd}\\
\end{array}\right.
\end{displaymath}
Hence $1+a\leq \frac{\dd}{2}$,
so that $1+a-\dd\leq -\frac{\dd}{2} < 0$.
We conclude 
$ 
h^0(X,L)<\frac{d}{2}+1
$ 
and we are done.

\noindent {\bf Case 2.}  {\it Assume $2g_1+1\leq d_1< 2g_1+\delta -1$.}

Set $d_1=2g_1+e_1$ where $1\leq e_1\leq \dd -2$. Hence
\begin{equation}
\label{v2g}
g_1=\frac{d_1-e_1}{2}.
\end{equation}
By Lemma~\ref{e} we have 
\begin{equation}
\label{v2d}
l\leq l_1+l_2-e_1-1.
\end{equation}
If $d_2\leq 2g_2$, then $l_2\leq \frac{d_2}{2}+1$. Using (\ref{v2g}) we have
$$
l \leq  d_1-g_1+1+\frac{d_2}{2}+1-e_1-1 
=d_1-\frac{d_1-e_1}{2}+\frac{d_2}{2}+1-e_1= 
\frac{d}{2}+1-\frac{e_1}{2}.
$$
Now $e_1\geq 1$ hence
$ 
l<\frac{d}{2}+1
$ 
and we are done. Also,   strict inequality  holds.

If $d_2\geq 2g_2+1$,
set  $d_2=2g_2+e_2$ with $e_2\geq 1$.
 We can also assume  $e_2\leq \dd -1$, otherwise we
are done by Case 1 (interchanging $C_1$ with $C_2$).

Now the situation is symmetric between $C_1$ and $C_2$, so up to switching them we may assume
$e_1\geq e_2$.
By Lemma~\ref{e} we have, 
$$
l\leq l_1+l_2-e_1-1 =d_1-g_1+1+d_2-g_2+1-e_1-1.
$$
Now, using  (\ref{v2g}) applied also to $C_2$
$$
l\leq d_1-\frac{d_1-e_1}{2}+1
+d_2-\frac{d_2-e_2}{2}+1-e_1-1=\frac{d}{2}+1+\frac{e_2-e_1}{2}.
$$
 As $e_1\geq e_2$ we conclude
$l\leq    \frac{d}{2}+1$. Moreover,  equality holds if $e_1=e_2$ and $l=l_1+l_2-e_1-1$.
\end{proof}

\begin{lemma}
\label{ae} Let $X$ be a semistable curve of genus $g\geq 2$, $d\leq 2g-2$ and $\md\in \BXd$.
Let $Z\subset X$ be a subcurve, set   $d_Z= 2g_Z+\dd_Z -1+a_Z $.
Then 
\begin{displaymath}
a_Z\leq \left\{ \begin{array}{l}
\frac{\delta_Z}{2}-1, \  \text{ if }  \  \dd_Z \text{ is even}\\
\\
\frac{\delta_Z-1}{2}-1, \text{ if }  \ \dd_Z \text{ is odd}\\
\end{array}\right.
\end{displaymath}
\end{lemma}
\begin{proof}
We just need to apply (\ref{BI}) and compute, using $d\leq 2g-2=w$:
$$
d_Z\leq \frac{dw_Z}{w}+\frac{\dd_Z}{2}
\leq w_Z+ \frac{\dd_Z}{2}=2g_Z-2+\dd_Z
+\frac{\dd_Z}{2}.
$$
Now the statement follows at once from
$$ d_Z=2g_Z+\dd_Z -1+a_Z \leq 2g_Z+\dd_Z -2 +\frac{\dd_Z}{2}.
$$
\end{proof}
\subsection{Clifford's Theorem in degree $2g-2$}

The following statement summarizes our results for   $d=2g-2$.
\begin{thm}
\label{cc}
Let $X$ be a connected curve of genus $g\geq 2$. Let $\md$ be a multidegree such that $|\md|= 2g-2$.
Assume that one of the following conditions hold.
\begin{enumerate}[(1)]
\item
\label{rcc}
$d_Z\geq 2g_Z-1$ for every  proper subcurve $Z\subsetneq X$.
\item
\label{bcc}
$X$ is semistable and 
 $\md $ is balanced.
\item
\label{ecc} $0\leq d_i\leq 2g_i$, for every $i=1,\ldots, \g$.
\end{enumerate}
Then 
$
h^0(X,L)\leq g 
$
for every $L\in \picX{\md}$.

Moreover, let $L\in \picX{\md}$ be such that  $h^0(X,L)= g $.
If (\ref{rcc}) or (\ref{bcc}) holds, or if (\ref{ecc}) holds with $\sep=\emptyset$,
then
 $L\cong\omega_X$.
\end{thm}

\begin{proof}
If assumption (\ref{rcc}) holds, then the theorem is proved in the subsequent Proposition~\ref{2g-2}.
Next,    (\ref{bcc}) implies (\ref{rcc}).
Indeed,  
$$
d_Z\geq w_Z\frac{d}{w}-\frac{\dd _Z}{2} = 2g_Z-2+\dd_Z -\frac{\dd _Z}{2}= 2g_Z-2+\frac{\dd _Z}{2}\geq 2g_Z - \frac{3}{2}.
$$
As $d_Z$ is an integer, we obtain  $d_Z\geq 2g_Z-1$.
 This settles the theorem under hypothesis (\ref{bcc}).

If (\ref{ecc}) holds, the fact that $h^0(L)\leq g$ is a special case of Proposition~\ref{ecl}.

Now let $L$ be such that $h^0(L)= g$. By Riemann-Roch and Serre duality this is equivalent to
\begin{equation}
\label{wL}
h^0( \omega_X\otimes L^{-1})=1.
\end{equation}
Now, $\deg \omega_X\otimes L^{-1}=0$ and we claim that $\mdeg \omega_X\otimes L^{-1}\geq 0$.
Indeed, as $X$ is free from separating nodes, for every $i=1,\ldots, \g $ we have $\dd_i\geq 2$.
Hence
$$
\deg_{C_i}\omega_X\otimes L^{-1}=2g_i-2+\dd_i-d_i\geq 2g_i-d_i\geq 0.
$$ 
Now, by Fact~\ref{cl00}, (\ref{wL}) is possible  if and only if $\omega_X\otimes L^{-1}\cong\O_X$, as claimed.
\end{proof}

\begin{example}
\label{ctcc} The   assumption  that $\sep$ be empty   is indeed necessary
in the last part of Theorem~\ref{cc}, as the present  example shows.
Let $X=C_1\cup C_2$ with $C_1\cdot C_2=1$. Then $2g-2=2g_1+2g_2-2$;
assume $g_1\geq 1$. Let $\md=(2g_1-2, 2g_2)$ and
$L=(\omega_{C_1}, L_2)$ for any $L_2\in \Pic^{2g_2}C_2$.
Then, as $\omega_{C_1}$ is free from base points, by Lemma~\ref{d1l} we have
$$
h^0(X,L)=h^0(C_1,\omega_{C_1})+h^0(C_2,L_{C_2})-1=g_1+(2g_2-g_2+1)-1=g.
$$
\end{example}

The following is a part of Theorem~\ref{cc}.
\begin{prop}
\label{2g-2}
Fix $X$ of genus $g$ and $\md$ such that
 $|\md|= 2g-2$; assume 
  $d_Z\geq 2g_Z-1$
for every   $Z\subsetneq X$.
Then for every $L\in \picX{\md}$ we have
\begin{equation}
\label{eqcl}
h^0(X,L)\leq g.
\end{equation}
If equality holds, then $L=\omega_X$.
\end{prop}
\begin{proof} 
We set $l=h^0(X,L)$, $l_i=h^0(C_i,L_i)$ and for any subcurve $Z\subset X$, $l_Z=h^0(Z,L_Z)$.
The hypothesis allows us to apply Lemma~\ref{er}, getting
\begin{equation}
\label{Step2}
l_Z=d_Z-g_Z+1 
\end{equation}
for every   $Z\subsetneq X$.

\noindent
{\bf Step 1.} {\it \  If there exists $i$ such that $d_i\geq 2g_i+\dd_i-1$
(in particular, $\md \neq \mdeg \omega_X$), then (\ref{eqcl}) holds
with strict inequality. }

Assume $d_1\geq 2g_1+\dd_1-1$. We can apply Lemma~\ref{e} to $X=C_1\cup Z$ where $Z= C_1^c$.
Using (\ref{Step2}) we  obtain
$$
l=l_1+l_Z-\dd_1=d_1-g_1+1+d_Z-g_Z+1-\dd_1=d-(g_1+g_Z+\dd_1-1)+1.
$$
Now, $g=g_1+g_Z+\dd_1-1$ hence $l=d-g+1=g-1<\frac{d}{2}+1$, as claimed.

\noindent
{\bf Step 2.} {\it \  If   $d_i\leq 2g_i+\dd_i-2$ for every $i$, then $\md = \mdeg \omega_X$. }

Set $d_i=2g_i+e_i$, then 
\begin{equation}
\label{}
\sum_{i=1}^{\g} e_i = 2(\dd -\g).
\end{equation}
This is trivial:   on the one hand $d=2g-2=\sum_{i=1}^{\g} (2g_i+e_i)$. On the other 
$2g-2=2\sum_{i=1}^{\g} g_i+2\dd -2\g$. So it suffices to compare the two identities.

Now, as $e_i\leq \dd_i-2$ by assumption, we have
$$
2(\dd -\g)=\sum_{i=1}^{\g} e_i\leq \sum_{i=1}^{\g}(\dd_i-2)=\sum_{i=1}^{\g}\dd_i-2\g=2\dd-2\g
$$
therefore equality must hold, which can only happen if $e_i= \dd_i-2$ for every $i$.
This is of course the same as saying $d_i=\deg_{C_i}\omega _X$, so we are done.

\noindent
{\bf Step 3.} {\it \  If   $d_i\leq 2g_i+\dd_i-2$ for every $i$, then the statement holds.}

By Step 2 the hypthesis is equivalent to $\md = \mdeg \omega_X$.
By Step 1 this is the only case that remains to be treated.
By Remark~\ref{ord} we can order the irreducible components of $X$  in such a way that for every $i\neq \g$ we have
\begin{equation}
\label{int}
C_i\cap (\cup_{j=i+1}^\g C_j)\neq \emptyset.
\end{equation}
Denote $\dd_{i,j}:=C_i\cdot C_j$ for every $i\neq j$.
Our choice of ordering of the $C_i$ yields   $ \sum_{j=1}^{i-1}\dd_{i,j}\leq \dd_i-1$, for all  $i<\g$. Therefore
(as $e_i+1=\dd_i-1$)
\begin{equation}
\label{comp}
\min\{e_i+1, \sum_{j=1}^{i-1}\dd_{i,j}\}= \min\{\dd_i-1,
\sum_{j=1}^{i-1}\dd_{i,j}\}=\sum_{j=1}^{i-1}\dd_{i,j},\  \  \  \forall i\neq \g .
\end{equation}
Now we shall bound $l$ by gluing one component at the time, starting with gluing $C_2$ to $C_1$ and ending
with gluing $C_{\g}$ to $\cup_{i=1}^{\g -1}C_i$. At each step we  apply Lemma~\ref{e}.
So, set $Z_i=\cup_{j=i}^iC_j\subset X$.
The first gluing (of $C_2$ to $C_1$) yields, using (\ref{comp}) and assuming $\g\geq 3$ 
(if $\g=2$ we jump to the last step, gluing $C_{\g}=C_2$ to $C_1$),
$$
h^0(Z_2,L_{Z_2})\leq l_1+l_2-\min\{e_2+1, \dd_{1,2}\}=l_1+l_2-  \dd_{1,2}.
$$
More generally, iterating up to the index $i\leq \g-1$, applying Lemma~\ref{e} and (\ref{comp}) at each
step, we obtain
\begin{equation}
\label{step}
h^0(Z_i,L_{Z_i})= l_1+\ldots+l_i-\dd_{1,2}-\ldots-\sum_{j=1}^{i-1}\dd_{i,j}.
\end{equation}
The last step is the gluing of $C_\g$, for which we need
\begin{equation}
\label{compg}
\min\{e_\g+1, \dd_\g\}=\min\{\dd_\g-1, \dd_\g\}=\dd_\g -1;
\end{equation}
hence
$$
l\leq h^0(Z_{\g-1},L_{Z_{\g-1}})+l_\g-\min\{e_\g+1, \dd_\g\}=h^0(Z_{\g-1},L_{Z_{\g-1}})+l_\g-\dd_\g +1.
$$
Combining everything we obtain
$$
l\leq \sum _{i=1}^\g l_i-\sum_{i=2}^{\g-1}(\sum_{j=1}^{i-1}\dd_{i,j})-\dd_\g+1= 
\sum _{i=1}^\g l_i-\dd+1=d-g+2=g 
$$
($\sum _{i=1}^\g l_i=d-\sum _{i=1}^\g g_i+\g=d-g+\dd +1$).

This finishes the proof of (\ref{eqcl}). Observe that
in our computation we had equality holding at every
step (see (\ref{step})) but the last one, when we glued $C_\g$. At that point, by (\ref{compg}), we are in the situation of
Lemma~\ref{e} (\ref{e!}). We obtain that equality holds for at most one $L$.
Now, if $L=\omega_X$, equality does hold, so this is the only case for which $h^0(X,L)=\frac{d}{2}+1=g$.
\end{proof}
We used the following simple   facts, which can be easily proved by induction.
\begin{remark}
\label{ord}
Let $X$ be a reducible, connected  curve.
\begin{enumerate}[(i)]
\item
\label{ordd} Then $X$ admits an irreducible component $C$ such that $C^c$ is connected
(such a $C$ will be called a {\emph {non-disconnecting component}}).
\item
\label{ordo} The  irreducible components $C_1,\ldots, C_\g$ of $X$ can be ordered  so that
for every $i<\g$ there exists $j>i$ such that $C_i\cap C_j\neq \emptyset$.
\end{enumerate}
\end{remark}

\section{Clifford's Theorem in low degree}
\subsection{Line bundles of degree at most $0$}

\begin{nota}
\label{Ysn} Let $X$ be  fixed. For any  $\md=(d_1,\ldots, d_{\g}) \in \Z^{\g}$, we  denote
\begin{equation}
\label{Zd}
\Zd:= \bigcup_{i:d_i<0}C_i\subset X.
\end{equation}

\begin{remark}
\label{cl-}
{\it Let $X$ be a nodal connected curve, and let $\md$ be such that $|\md|<0$
 and  $\md \leq 0$. Then for every $L\in \picX{\md}$ we have $h^0(X,L)=0$.}

Indeed $h^0(\Zd, L_{\Zd})=0$, of course. Now, for any connected component, $Y$, of $\ov{X\smallsetminus \Zd}$,
we have $\md_Y=(0,\ldots,0)$, hence   $h^0(Y,L_Y)\leq 1$ with equality  if and only if $L_Y=\O_Y$, in which case $L_Y$ has no base
points. So the remark follows from Lemma~\ref{d1l}.
\end{remark}

Fix $L\in \picX{\md}$; for every nonzero $s\in H^0(X,L)$ we let $Y_s$ be the subcurve of $X$ where $s$ does not vanish,
and $\Ws$ its complementary curve:
\begin{equation}
\label{Ys}
Y_s:=\bigcup_{i:s_{|C_i}\neq 0} C_i  \subset X \ \   \text{  and } \  \Ws:=\ov{X\smallsetminus \Ys}.
\end{equation}
Note that   $\Zd\subset \Ws$ and  $\md_{\Ys}\geq 0$.
\begin{remark}
\label{S1}
{\it With the above notation, fix
 $\md$   such that $\md \not\geq 0$, and  let $L\in \picX{\md}$. For every nonzero $s\in H^0(X,L)$
(if it exists)
we have  $d_{\Ys}\geq \dd_{\Ys}$.}

Indeed $\Zd$ is nonempty, hence $\Ws$ is nonempty. Since $s$ vanishes on $\Ws\cap \Ys$ the claim follows.
\end{remark}
\end{nota}
\begin{lemma}
\label{cl0}
Let  $X$ be a semistable curve,   $d\leq 0$ and $\md \in \BXd$. 

Then for every $L\in \picX{\md}$, with $L\neq \O_X$, we have $h^0(X,L)=0$.
\end{lemma}
\begin{proof}
If $\md=(0,\ldots,0)$ the statement follows from Fact~\ref{cl00}.
We can thus assume $\md \not\geq 0$.
As $\md$ is balanced, for every subcurve $Z\subset X$ we have 
$$
d_Z\leq \frac{dw_Z}{w}+\dfZ\leq \dfZ .
$$
Hence $d_Z<\dd_Z$. Combining this with Remark~\ref{S1},   we  are done.
\end{proof}
\begin{nota}
By  Riemann-Roch and Serre duality, any statement about sections of line  bundles of degree $2g-2$
has a  dual  statement about sections of line  bundles of degree $0$.
The following is the dual of Theorem~\ref{cc}.
\begin{thm}[Clifford for $d=0$]
\label{c0}
Let $X$ be a  curve of genus $g\geq 2$. Let $\md$ be such that $|\md|= 0$.
Assume that one of the following conditions hold.
\begin{enumerate}[(1)]
\item
\label{rc0}
$d_Z\leq \dd_Z-1$ for every  proper subcurve $Z\subsetneq X$.
\item
\label{bc0}
$X$ is semistable and 
 $\md $ is balanced.
\item
\label{ec0} $\dd_i\leq d_i\leq 2g_i-2+\dd_i$, for every $i=1,\ldots, \g$.
\end{enumerate}
Then 
$
h^0(X,L)\leq 1 
$
for every $L\in \picX{\md}$.

Moreover, let $L\in \picX{\md}$ be such that  $h^0(L)= 1$.
If (\ref{rcc}) or (\ref{bcc}) holds, or if (\ref{ecc}) holds with $\sep=\emptyset$,
then
 $L\cong\O_X$.
\end{thm}
\begin{proof} This follows from Theorem~\ref{cc}, applying Riemann-Roch and Serre duality,
together with some trivial arithmetic.
\end{proof}
\end{nota}

\subsection{Clifford's theorem in degree at most $4$}

The main result of this section is Theorem~\ref{cl4}, stating   the Clifford inequality   for line bundles of balanced
multidegree on semistable curves free from separating nodes.
The proof is organized as follows.  In Lemma~\ref{lemmino}, Lemma~\ref{lemmino2} and Proposition~\ref{lemmino3} we treat  the case
$\md \geq 0$, 
without assuming that $\md$ is balanced.
The proof of Theorem~\ref{cl4} is thus reduced to assume that $\md$ has some negative entry.

Quite interestingly, if $d\geq 5$ Clifford's theorem fails  even when $X$ has no separating nodes. See Example~\ref{cl5no}.

\begin{nota}
\label{sl}
Let $n\in \sep$ be a separating node of $X$; then there exist two   subcurves $Z_1$ and $Z_2$ of $X$ such that
$X=Z_1\cup Z_2$ and $n=Z_1\cap Z_2$.
Such curves $Z_1$, $Z_2$ are called the  tails  of $X$ generated by $n$. So, a subcurve $Z\subset X$ is
called a {\it tail} if $Z\cdot Z^c=1$. As $X$ is connected,   its tails are  connected.

Let $C\subset X$ be a subcurve.  $C$ is called a {\it separating line} if $C\cong \pr{1}$ and if $C$ meets
its complementary curve $C^c$ only in separating nodes of $X$.
Equivalently:  a separating line $C\subset X$ is a smooth rational component such that $C^c$ has
a number of connected components equal to $C\cdot C^c$.

If $X\cong\pr{1}$, then  $X$ is a separating line of itself.

If $Y$ is a disconnected curve and $C\subset Y$, we say $C$  is a separating line of $Y$ if it is so for the
connected component of $Y$ containing $C$.

Observe that if $C$ is a separating line, we have
\begin{equation}
\label{ZC}
Z\cdot C \leq 1\  \    \text{for every connected }\  Z\subset C^c.
\end{equation}
\end{nota}
\begin{remark}
\label{snr}
Assume $\sep =\emptyset$;   equivalently, assume that $X$ has no tails.
Let $Z$ be a  subcurve of $X$.
  \begin{enumerate}[(A)]
\item
\label{snrm}
 If $m$ is the number of connected components of $Z$, then $m\leq \frac{\dd_Z}{2}$
\item
\label{snrC}
Let $X=D\cup Y$ with $D$ connected. If $C\subset Y$ is a separating line of $Y$, then
 $\ov{X\smallsetminus C}$ is connected.
\end{enumerate}
The only statement that is not  obvious is (\ref{snrC}). 
Let $Y_1,\ldots , Y_m$ be the connected components of $Y$ and suppose $C\subset Y_1$.
We can assume  $C\neq  Y_1$. Thus
every connected component of $\ov{Y_1\smallsetminus C}$ is a tail of $Y_1$; as $X$ has no tails
   $D$ intersects
every connected component of $\ov{Y_1\smallsetminus C}$. On the other hand, $D$ obviously intersects $Y_i$ for all
$i\geq 2$, therefore 
$\ov{X\smallsetminus C}$  is connected.
\end{remark}

\begin{lemma}
\label{lemmino}
Let  $L\in \picX{\md}$. 
Assume $\md=(1,0,\ldots,0)$.
Then either  
$ h^0(X,L)\leq 1,$ 
or $C_1$ is a separating line,    $h^0(X,L)=2$ and  $L_{C_1^c}=\O_{C_1^c}$.
\end{lemma}

\begin{proof}
Denote $Y=C_1^c$ and let $Y=\coprod_{i=1}^cY_i $ be the decomposition  into connected components.
Of course $C_1$ must intersect every $Y_i$.

If   $g_1\geq 1$ we have $h^0(C_1, L_{C_1})\leq 1$ hence
 the lemma follows from  Remark~\ref{ur} (with $V=C_1$). So it suffices to assume $C_1\cong \pr{1}$.
If  $C_1$ is not a separating line   there exists at least one connected component of $Y$, $Y_1$ say, such that $ C_1\cdot Y_1\geq 2$.
Set $X_1=C_1\cup Y_1$, then by Remark~\ref{ur} and  Lemma~\ref{e} we conclude as follows
$$
h^0(X,L)\leq h^0(X_1,L_{X_1})\leq h^0(C_1, L_1)+h^0(Y_1,L_{Y_1})-2\leq 2+1-2=1.
$$ 
If $C_1$ is a separating line and for some component of $Y$, $Y_1$ say, we have $L_{Y_1}\neq\O_{Y_1}$, then every section
of $L$ has to vanish on $Y_1$, hence not every section of $\O_{C_1}(1)$ extends to a section of $L$.

Conversely, if $L_{Y_i}= \O_{Y_i}$ for all $i$  it is obvious that $h^0(X,L)=2$.
\end{proof}

\begin{lemma}
\label{lemmino2}
Let  $L\in \picX{\md}$.
Assume that $|\md|=2$ and  $\md \geq 0$.
Then  either
$ h^0(X,L)\leq 2,$ 
or $h^0(X,L)=3$ and one of the following cases occurs
\begin{enumerate}[(i)]
\item
\label{}
$\md =(2,0,\dots, 0)$ with    $C_1$   a separating line.
\item
\label{}
$\md =(1,1,0,\dots, 0)$, with    $C_1$ and $C_2$   separating lines.
\end{enumerate}
\end{lemma}
\begin{proof}
Assume $h^0(L)\geq 3$. 
For every nonsingular point $p$ of $ X$ we have
\begin{equation}
\label{Lp}
h^0(L(-p))\geq  h^0(L)-1\geq 2.
\end{equation}
Of course, $\deg L(-p)=1$ and,
if   $p$ lies in a component $C_1$ such that $d_1>0$  we have
  $\mdeg L(-p)\geq 0$. 
By Lemma~\ref{lemmino} we get $h^0(L(-p))\leq 1$, 
unless $X$ has a separating line $E$  with   $\deg_EL(-p)=1$. If $X$ does not have such a separating line we got a
contradiction to  (\ref{Lp}). 
Now,  $X$   admits such a separating line $E$ if and only if
either  $d_1=2$ and $E=C_1$,
or $d_1=1$, hence $d_2=1$, and $C_2$ is a separating line. By placing $p\in C_2$ we get that both $C_1$ and $C_2$
are separating lines.
By Lemma~\ref{lemmino}   $h^0(L(-p))=2 $, so $h^0(L)=3$ by (\ref{Lp}) and we are done.
\end{proof}
\begin{prop}
\label{lemmino3}
Let $X$ be a   stable curve free from separating nodes.
Let $\md$ be such that $\md \geq 0$ and $|\md|=3,4$.
Then $h^0(X,L)\leq |\md|/2 +1
$
   for every $L\in \picX{\md}$.
\end{prop}
\begin{remark}
The hypotheses $X$ stable and $\sep=\emptyset$ are necessary, as shown by Examples \ref{cl3no} and \ref{cl3n}.
\end{remark}
\begin{proof}
We first treat the case $|\md|=3$.
Consider the irreducible component $C_1$ of $X$;
we shall denote $C_1^c=Y_1\coprod\ldots \coprod Y_m$   the connected component decomposition.
Observe that for 
every   $Y_i$  we have
$Y_i\cdot C_1\geq 2$. 
We set $$
X_1:=C_1\cup Y_1\subset X.
$$ 
We shall repeatedly apply Lemma~\ref{e} and Remark~\ref{ur}.
 
\noindent
{\bf Case 1:} $\md =(3,0,\ldots,0)$.
We have $h^0(X,L)\leq h^0(X_1,L_{X_1})$ by
Remark~\ref{ur}.
Hence it suffices to assume that $C_1$ has  genus  $g_1\leq 1$.

If $g_1=1$,  
by the initial observation and Lemma~\ref{e} we have
$h^0(X_1,L_{X_1})\leq 3+1-2=2$ and we are done.

If $C_1\cong \pr{1}$ we have $h^0(C_1,L_1)=4$ and  $C_1\cdot C_1^c\geq 3$.
If $C_1^c$ has a connected component,
$Y_1$,  such that $C_1\cdot Y_1\geq 3$,
then 
$h^0(Y_1, L_{Y_1})\leq 1$.   
By Lemma~\ref{e} we get $h^0(X_1,L_{X_1})\leq 4+1-3=2$, as wanted. 

Let now 
 $C_1\cdot Y_i= 2$ for all $i=1,\ldots ,m$.
Set 
$X_2=Y_1\cup Y_2\cup C_1\subset X$.
Then  $C_1\cdot (Y_1\cup Y_2)\geq 4=d_1+1$, hence
by  Lemma~\ref{e},
$$
h^0(X_2,L_{X_2})\leq h^0(C_1,L_1)+h^0(Y_1,L_{Y_1})+h^0(Y_2,L_{Y_2})-4\leq 4+2-4=2.
$$ 
By 
Remark~\ref{ur} we are done.

\noindent
{\bf Case 2:} $\md =(1,2,0,\ldots,0)$.

Denote 
$l_i=h^0(C_i, L_i)$.
Assume $C_1^c$ connected;  by Lemma~\ref{lemmino2},  $h^0(C_1^c, L_{C_1^c})\leq 3$ and equality holds  if and only if $C_2$ is a separating line
of 
$C_1^c$. 
If this is not the case,  by Lemma~\ref{e} and   $\dd_1\geq 2$, we get 
$h^0(X,L)\leq l_1+2-2\leq 4-2=2$, as wanted.

If $C_2$ is a separating line of 
$C_1^c$,  then $l_2=3$, and $C_2^c$ is connected, by Remark~\ref{snr} (\ref{snrC}); hence $h^0(C_2^c,L_{C_2^c})\leq 2$.
Since $\dd_2\geq 3$ (as $d_2=2$) we obtain
$$
h^0(X,L)\leq l_2+h^0(C_2^c,L_{C_2^c})-3\leq 5-3=2
$$  and we are done. This part  works
regardless of $C_1^c$ being connected.

Now   let $C_1^c$ have  $m\geq 2$ connected components. 
We can assume that $C_2$ is not a separating line of $C_1^c$.
Let $C_2\subset Y_1$; we have $h^0(Y_1,L_{Y_1})\leq 2$. By Lemma~\ref{e} we get
$h^0(X_1,L_{X_1})\leq h^0(C_1,L_1)+h^0(Y_1,L_{Y_1})-2\leq 2$. By Remark~\ref{ur} we are done.

\noindent
{\bf Case 3:} $\md =(1,1,1,0,\ldots,0)$.
By Proposition~\ref{ecl} we may assume that $C_1\cong \pr{1}$.
Moreover, by Lemma~\ref{1sp}, up to permuting the first three components,
we can assume that $C_2$ and $C_3$ are not   separating lines of $C_1^c$.
If $C_1^c$ is connected, by Lemma~\ref{lemmino2} we have $h^0(C_1^c, L_{C_1^c})\leq 2$
(as $C_2$, $C_3$ are not  separating lines of $C_1^c$).
By  Lemma~\ref{e} we have 
$h^0(X,L)\leq h^0(C_1, L_1)+h^0(C_1^c, L_{C_1^c})-2\leq 2+2-2\leq 2$ and we are done. 

Now assume $C_1^c$ has  $m\geq 2$ connected components.  
If $C_2\cup C_3$ lies in one connected component, $Y_1$, then $h^0(Y_1, L_{Y_1})\leq 2$
(just as above). Therefore 
$ 
h^0(X,L)\leq h^0(X_1,L_{X_1})\leq   2+2-2=2
$ ($X_1=C_1\cup Y_1$).
If instead $C_2$ lies in $Y_1$ and $C_3$ lies in $Y_2$, then for $i=1,2$ we have $h^0(Y_i, L_{Y_i})\leq 1$
by  Lemma~\ref{lemmino} (as $C_2$, $C_3$  are  not separating lines of, respectively,  $Y_1$, $Y_2$). We conclude 
$h^0(X_1,L_{X_1})\leq 2+1-2=1$. 
Now, let 
$ 
X_2=X_1\cup Y_2,$ then 
$$h^0(X,L)\leq h^0(X_2,L_{X_2})\leq h^0(X_1,L_{X_1})+h^0(Y_2, L_{Y_2})\leq 2.
$$
The proof for $d=3$ is complete.

\

Now let $|\md|=4$.  By contradiction, suppose that $h^0(X,L)\geq 4$. As $\md \geq 0$, there exists a component, $C_1$ say, such
that $d_1\geq 1$. Let $p\in C_1$ be a nonsingular point of $X$,
then $h^0(L(-p))\geq h^0(L)-1\geq 3$.
Now, 
 $\deg L(-p)=3$ and $\mdeg L(-p)\geq 0$. By the previous part, $h^0(L(-p))\leq 2$; impossible.
\end{proof}

In the proof we used the following combinatorial Lemma.
\begin{lemma}
\label{1sp} Let $X$ be stable, $\sep=\emptyset$,and $C_1, C_2$ two irreducible components of $X$.
Assume  $C_2$ is a separating line of $C_1^c$, and  $C_1$ is a separating line of $C_2^c$
(i.e. $(C_1,C_2)$ is a $\B$-pair, see definition~\ref{bp}).
Then for every other component $D$ of $X$, $C_1$ and $C_2$ are not separating lines of $D^c$.
\end{lemma}
\begin{proof}
Note that by Remark~\ref{snr} (\ref{snrC}), $C_1^c$ and $C_2^c$ are connected.
Call  $T_1,\ldots, T_t$ the tails of $C_1^c$   generated by $C_2$. Thus $C_1^c=C_2\cup T_1\cup\ldots \cup T_t$,
with $T_i\cap T_j=\emptyset$ and $T_i\cdot C_2=1$.  As $C_2^c$ is connected, $C_1$ must intersect every $T_i$.
As $C_1$ is a separating line of $C_2^c$, we have
\begin{equation}
\label{C1T}
C_1\cdot T_i=1,\  \  \forall i .
\end{equation}
Let $D$ be another component of $X$,  assume $D\subset T_1$. Set $Z=C_2\cup T_2\cup\ldots \cup T_t$,
so that $C_1^c=Z\cup T_1$, hence
$ 
\dd_{C_1}=Z\cdot C_1+T_1\cdot C_1 = Z\cdot C_1 +1 \geq 3,
$ 
by (\ref{C1T}) and the stability of $X$.
We conclude $Z\cdot C_1   \geq 2$. This implies that $C_1$ cannot be a separating line of $D^c$, as $Z$ is
connected and 
$Z\subset D^c$ (cf. \ref{sl} (\ref{ZC})). The same argument with $C_1$ and $C_2$ switching roles yields that 
$C_2$ is not a separating line of $D^c$.
\end{proof}

\begin{thm}
\label{cl4} Let $X$ be a stable curve free from separating nodes. Let $\md$ be balanced with $0<|\md |\leq 4$; let
 $L\in \picX{\md}$. Then
\begin{enumerate}[(i)]
\label{}
\item 
$ h^0(X,L)\leq  |\md| / 2 +1.$ 
\item
\label{=+}
If $|\md|=1,2$ and  $h^0(X,L)= |\md|$, then $\md\geq 0$.
\end{enumerate}
\end{thm}
If  $|\md|=1,2$  the hypotheses on $X$ can be weakened as follows.
\begin{add}
\label{add}
If $|\md|=1$ the same holds if $X$ is semistable and has no separating lines.
If $|\md|=2$ the same holds   if $X$ is semistable and $\sep=\emptyset$.
\end{add}
\begin{proof}

If $\md\geq 0$ the
statement follows from    Lemmas~\ref{lemmino}, \ref{lemmino2} and Proposition~\ref{lemmino3}. 
So, assume $\md\not\geq 0$; set   $d=|\md|$.
We shall inductively define a useful subcurve $V\subseteq X$.
Let $V_0:=\Zd$ (see (\ref{Zd})).
Now define $V_1\subset X$ 
$$
V_1:= V_0\cup \bigcup_{C_i\cdot V_0>d_i=0}C_i;
$$
so $V_1$ is the union of $V_0$ with all components of degree $0$ which intersect $V_0$.
 Next
$$
V_2:=V_1\cup \bigcup_{\stackrel{C_i\not\subset V_1, d_i\leq 1,}{ C_i\cdot V_1>d_i}}C_i.
$$
Iterating
$$
V_{h+1}:=V_h\cup \bigcup_{\stackrel{C_i\not\subset V_h, d_i\leq h,}{ C_i\cdot V_h>d_i}}C_i \subset X.
$$
 
Of course, $V_0\subseteq V_1\subseteq \ldots \subseteq V_h\subseteq V_{h+1} \subseteq  \ldots \subseteq X$, therefore
there exists an   $m\geq 0$ minimum for which 
$V_n= V_m$ for every   $n \geq m$.
We set $V:=V_m$.

W e claim that  every $s\in H^0(X,L)$   vanishes identically on $V$.
It is clear that $s$ vanishes on $V_0$; let us prove the claim inductively.
Let $h\geq 0$ be such that $V_{h+1}$ is not equal to $V_h$; by   induction 
 $s$   vanishes identically on $V_h$.
Let $C\subset V_{h+1}$ be such that   $C$ is not contained in $V_h$. Then
  $s$ vanishes on $C\cap V_h$.
Now, $V_{h+1}$ is constructed  so that $C\cdot V_h>\deg_CL>0$, therefore $s$ vanishes on $C$.
The claim is proved.

If $V=X$  we have $H^0(X,L)=0$ and we are done.
So assume that   $Y:=V^c$ of $V$ is not empty.
Denote $G_Y\in \Div Y$ the  divisor cut out by $V$, so that 
\begin{equation}
\label{dG}
\deg G_Y=\dd_Y.
\end{equation}
Notice that 
\begin{equation}
\label{hG}
H^0(X,L)\cong H^0(Y, L_Y(-G_Y)).
\end{equation}
By   construction we have
\begin{equation}
\label{mdY}
\md_Y-\mdeg G_Y\geq 0. 
\end{equation}
We  claim that
\begin{equation}
\label{aG}
0\leq d_Y-\dd_Y\leq d-2.
\end{equation}
Set $a=d_Y-\dd_Y$. That $0\leq a$ follows from (\ref{dG}) and (\ref{mdY}). 
Now, notice that $w_Y<w$. Indeed, as $\md_V\not\geq0$ by construction,
$V=Y^c$ is not a union of exceptional components (see the initial observation). Hence 
(cf. \ref{w}) $w_V>0$ and $w_Y=w-w_V<w$.  
As $\md$ is balanced, we obtain
\begin{equation}
\label{BII}
\dd_Y+a=d_Y\leq \frac{\dd_Y}{2}+\frac{dw_Y}{w}<\frac{\dd_Y}{2}+d.
\end{equation}
Therefore
$\dd_Y\leq 2d-2a-1$.  As $\sep =\emptyset$ we have $\dd_Y\geq 2$.
We obtain
$$
2d-2a-1\geq 2
$$
hence $a\leq d-3/2$, so that $a\leq d-2$. (\ref{aG}) is proved.

We continue the proof with a case by case analysis. 
 
Case $d=1$.  The inequality (\ref{aG}) makes no sense, hence   $Y$ is empty,
i.e. $h^0(L)=0$. 
We conclude that   if $h^0(L)\neq 0$, 
then $\md\geq 0$, a case treated in Lemma~\ref{lemmino}.
The assumptions  $X$ stable and $\sep=\emptyset$ 
can clearly be weakened by, respectively, $X$ semistable, and
containing no
separating line (needed for  Lemma~\ref{lemmino}).   If $d=1$   the Theorem and the Addendum are proved.

Case $d=2$. By (\ref{aG})  we have  $d_Y=\dd_Y$, hence 
$\deg L_Y(-G_Y)=0$. Now, using (\ref{BII}) we get
$
\dd_Y=d_Y<\frac{\dd_Y}{2}+2
$, hence $\dd_Y\leq 3$. This yields that $Y$ is connected, by Remark~\ref{snr} (\ref{snrm}).
We can  apply Fact~\ref{cl00} to $L_Y(-G_Y)$, obtaining, with (\ref{hG}),
$$
h^0(X,L)=h^0(Y,L_Y(-G_Y)\leq 1.
$$
This concludes the proof if $d=2$. We also showed that if $h^0(X,L)=2$ then $\md\geq 0$.
Observe that the argument works if $X$ is semistable, so the Theorem and the Addendum are proved.
The remaining cases will be treated similarly.
 
Case $d=3$. By (\ref{aG}) we have two possibilities: either $\dd_Y=d_Y$ or $\dd_Y+1=d_Y$.
If $\dd_Y=d_Y$  we have, using (\ref{BII}), 
$
\dd_Y=d_Y<\frac{\dd_Y}{2}+3
$, hence $\dd_Y\leq 5$. Therefore $Y$ has at most two connected components
(by Remark~\ref{snr} (\ref{snrm})). Let $Y_i$ be a connected component of $Y$, then, by (\ref{mdY}),
 $d_{Y_i}=\dd_{Y_i}$, and   we can apply Fact~\ref{cl00} to $L_{Y_i}(-G_{Y_i})$ (with self-explanatory notation).
Hence $h^0(Y_i,L_{Y_i}(-G_{Y_i})\leq 1$; now $Y$ has at most two connected components, hence by (\ref{hG}) we obtain
$h^0(X,L)\leq 2$.

If $d_Y=\dd_Y+1$,
by (\ref{BII}) 
$
\dd_Y+1=d_Y<\frac{\dd_Y}{2}+3
$, hence  $\dd_Y\leq 3$, so  $Y$ is connected.
By (\ref{mdY}) and (\ref{aG}) we can apply Lemma~\ref{lemmino} to $L_Y(-G_Y)$; we get
$$
h^0(X,L)=h^0(Y,L_Y(-G_Y)\leq 2.
$$ 
This finishes the proof in case $d=3$.

Case $d=4$.
By (\ref{aG}) we have three possibilities:   $d_Y=\dd_Y$, $d_Y=\dd_Y+1$ or  $d_Y=\dd_Y+2$.

If $d_Y=\dd_Y$, we get
$
\dd_Y=d_Y<\frac{\dd_Y}{2}+4
$, hence $\dd_Y\leq 7$. Therefore $Y$ has at most three connected components
(again by Remark~\ref{snr} (\ref{snrm})).
Arguing as in the analogous case when $d=3$ ($d_Y=\dd_Y$) we see that 
$h^0(X,L)\leq 3$ so we are done.

If $d_Y=\dd_Y+1$,
by (\ref{BII}) 
$
\dd_Y+1=d_Y<\frac{\dd_Y}{2}+4
$, hence  $\dd_Y\leq 5$ and  $Y$ has at most two connected components.
If $Y$ is connected arguing as in the analogous case when $d=3$
we conclude 
$h^0(X,L)\leq 2$ and we are done.
If $Y$ has two connected components, $Y_1$ and $Y_2$, then we have $d_{Y_1}=\dd_{Y_1}$ and 
$d_{Y_2}=\dd_{Y_2}+1$. We can therefore apply Fact~\ref{cl00}
to get $h^0(Y_1, L_{Y_1}(-G_{Y_1}))\leq 1$, and 
\ref{lemmino} to get $h^0(Y_2, L_{Y_2}(-G_{Y_2}))\leq 2$. Summing up we obtain
$$
h^0(X,L)= h^0(Y_1, L_{Y_1}(-G_{Y_1}))+h^0(Y_2, L_{Y_2}(-G_{Y_2}))\leq 3
$$
and we are done.
Finally, if $d_Y=\dd_Y+2$, by the usual argument we get $\dd_Y\leq 3$ hence $Y$ is connected.
By  Lemma~\ref{lemmino2} we have $3\geq h^0(Y, L_Y(-G_Y))=h^0(X,L)$ and we are done.
\end{proof}
\subsection{Counterexamples}
\label{cex}
\begin{example}
\label{P1}
{\it Failure of Clifford's theorem:   $d=1$,  $\md\geq 0$ balanced ($X$ contains a separating line).}
Let $X=C_1\cup C_2\cup C_3\cup C_4$ with, for $i,j\geq 2$,
$C_i\cap C_j=\emptyset$ and $C_1\cdot C_i =1$ (the  dual graph of $X$ is in  Figure~\ref{efig1}).
Assume $C_1=\pr{1}$ (hence $C_1$ is a separating line) and $g_i=h\geq 1$ (hence $X$ is   stable).
Thus $g=3h$ and $w=6h-2$. Set
$\md =(1,0,0,0)$,
one checks that $\md \in \BX1$.
Let
$$
L:=(\O_{C_1}(1), \O_{C_2},\O_{C_2},\O_{C_4}).
$$
Then, as all $L_i$ are free from base points, we get
$ 
h^0(X,L)=\sum_1^4h^0(C_i,L_i)-3=2.
$ 
\begin{figure}[!htp]

$$\xymatrix@=1pc{
&&*{\bullet}  \ar@{-}[drr] ^(.00001){C_1}^(.99){C_4}  \ar@{-}[dll] _(.999){C_2} \ar@{-}[d]^(.8){C_3}\\
*{\bullet}  && *{\bullet}&&*{\bullet}
}$$
\caption{Dual graph of the curve in Example~\ref{P1}.}
\label{efig1}
\end{figure}
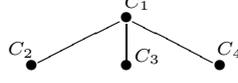
\end{example}

\begin{example}
\label{P2}
{\it $\Cl L=0$  with $\mdeg L \in \BX1$,   $\mdeg L\not\geq 0$
($\sep \neq \emptyset$).}
Let $X=C_1\cup C_2\cup C_3$ with, 
$C_1\cdot C_2=2$,\   $C_2\cdot C_3 =1$ and $ C_1\cap C_3=\emptyset$ 
(see the picture below).
Thus $n=C_2\cap C_3$ is a separating node;
for $i=2,3$, write  $q_i\in C_i$ the point corresponding to this node.
Assume $g_1=g_2=1$ and $g_3=4$, thus $g=7$.
 Set
$\md =(1,-1,1)$;
one checks that $\md \in \BX1$. Call $Z=C_1\cup C_2\subset X$ and
let $L_{1,2}\in \Pic^{(1,-1)}Z$ be arbitrary. Note that $h^0(Z,L_{1,2})=0$. Set
$$
L:=(L_{1,2},\O_{C_3}(q_3)).
$$
Then, as $L_{1,2}$ and $\O_{C_3}(q_3)$ both have a base point in  the respective  branch ($q_2$ and $q_3$) of $n$, we get
$ 
h^0(X,L)=h^0(Z,L_{1,2})+h^0(C_3,\O_{C_3}(q_3))=1
$. 
\end{example}

\begin{figure}[!htp]
$$\xymatrix@=1pc{
*{\bullet}   \ar@{-}@/_/[rr]_(.001){C_1}_(.99){C_2} \ar@{-}@/^/[rr]
&&*{\bullet} \ar@{-}[rr] _(.999){\ C_3}&&*{\bullet}
}$$
\caption{Dual graph of the curve in Example~\ref{P2}.}
\label{fig2}
\end{figure}
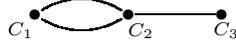

\begin{example}
\label{cl2no}
{\it Failure of Clifford's theorem: $d=2$, $\md $ balanced ($\sep \neq\emptyset$).}
Let $X=C_1\cup C_2\cup C_3\cup C_4$ with, for $i,j\geq 2$,
$C_i\cap C_j=\emptyset$ and $C_1\cdot C_i =1$ (same dual  graph as in Figure~\ref{efig1}). Let $g_1=1$ and $g_2=g_3=g_4=3$ so that  $g=10$.
Let $\md =(-1,1,1,1)$; one checks that $\md$ is the unique balanced multidegree
of degree $2$.
Let $L_1$ be any line bundle of degree $-1$ on $C_1$. For $i=2,3,4$ denote by $q_i\in C_i$
the point corresponding to the node $C_1\cap C_i$. Consider the degree $2$ line bundle on $X$
$$
L=(L_1,\O_{C_2}(q_2),\O_{C_3}(q_3),\O_{C_4}(q_4)).
$$
As every section of $\O_{C_i}(q_i)$ vanishes in $q_i$, we get that $H^0(X,L)=3$.
\end{example}

\begin{example}
\label{cl3no}
{\it Failure of Clifford's theorem: $d\geq 3$, $\md $ balanced, $\sep =\emptyset$ ($X$ strictly semistable).}
For $d\geq 3$ consider the curve
$X=C_1\cup\ldots\cup C_{2d}$ whose dual graph is a $2d$-cycle, i.e. a closed polygon with $2d$ vertices,
$C_1,\ldots, C_{2d}$. We set $C_i\cdot C_{i+1}=C_{2d}\cdot C_1=1$ for all $i\geq 1$ and
 $C_i\cdot C_j=0$ for all other intersections. So $X$ has $2d$ nodes.
Let $C_{2i-1}\cong\pr{1}$ for all $i$, so that the odd indexed components are exceptional; now let  
 all the even indexed components  be  smooth of genus $1$. Therefore
$g=d+1$.
Now  choose
the multidegree $\md=(1,0,1,\ldots, 1,0)$ and set
$L_{C_{2h}}\cong \O_{C_{2h}}$ for all $h$  (of course $L_{C_{2h+1}}\cong \O_{\pr{1}}(1)$). 
One easily checks that $\md$ is balanced.
It is also clear that for any $L\in \Pic X$ whose restrictions to the $C_i$ are as above,
we have
$ 
h^0(X,L)\geq 2d+d-2d=d.
$ 
So Clifford's inequality fails.
\end{example}

\begin{example}
\label{cl3n}
{\it Failure of  Clifford's theorem: $d\geq 3$, $\md \geq 0$, $\sep \neq\emptyset$.}
Let $X=C_1\cup C_2\cup  C_3$ with $C_1$ of genus $1$ and $g_i\geq 1$.
Let $C_1\cdot C_2=C_1\cdot C_3=1$ and $C_2\cdot C_3=0$
(the dual graph of $X$ is obtained from the graph in Figure~\ref{efig1} by removing the vertex $C_4$ and the edge adjacent to it). Let $L=(L_1, \O_{C_2}, \O_{C_3})\in \Pic^d X$
with $\deg L_1=d$. Then $h^0(L)=d$.
\end{example}

\begin{example}
\label{cl5no}
{\it Failure of  Clifford's theorem:     $d=5$, \ $\md $ balanced and   $\sep =\emptyset$.}
Let $X=C_1\cup C_2\cup C_3\cup C_4\cup C_5$ with, for $i,j\geq 2$,
$C_i\cap C_j=\emptyset$ and $C_1\cdot C_i =2$ for all $i\geq 2$. 
So every node of $X$ lies on $C_1$, and $\delta =8$ (the dual graph of $X$ is in Figure \ref{ex5}).
Now let $h$ be any nonnegative integer.
Let $C_1$ be of genus $g_1=h$, and let $C_i$ have  genus $h+3$ for every $i\geq 2$.
Hence $g=5h+16$.
We now pick $d=5$ and 
$ \md =(-3,2,2,2,2).$ 
It is straightforward to check that  $\md$ is balanced.

Now  for $i\geq 2$,  set $\{p_i,q_i\}=C_1\cap C_i\subset C_i$.
Let $L$ be any line bundle whose restrictions $(L_1,\ldots, L_5)$ 
are as follows. $L_1\in \Pic^{-3}C_1$ is arbitrary, while
$L_i=\O_{C_i}(p_i+q_i)$, for   $i=2,3,4,5$.

Now every section $s$ of $L$  vanishes identically on $C_1$, hence $s$ vanishes on $p_i, q_i$.  
Conversely, any quadruple of sections $s_i\in H^0(C_i,L_i(-p_i-q_i))$, for $i=2,\ldots,5$, glues to  a section of $L$.
We
conclude
$ 
h^0(X,L)= \sum_{i=2}^5h^0(C_i,L_i(-p_i-q_i))=4.
$ 
So $L$ violates Clifford inequality.
Similar examples exist for higher degree $d$.

\begin{figure}[!htp]
$$\xymatrix@=1pc{
*{\bullet}&& &&*{\bullet} 
\\
 && *{\bullet}
 \ar@{-}@/_/[ull]_(.2){C_1}_(.9){C_2}    \ar@{-}@/^/[dll]^(.9){C_3}    \ar@{-}@/^/[ull]   \ar@{-}@/_/[dll] 
 \ar@{-}@/_/[urr]    \ar@{-}@/^/[drr]  \ar@{-}@/^/[urr]^(.9){C_4}   \ar@{-}@/_/[drr] _(.9){C_5}  
\\
*{\bullet}&& &&*{\bullet}
}$$
\caption{Dual graph of the curve in Example~\ref{cl5no}.}
\label{ex5}
\end{figure}
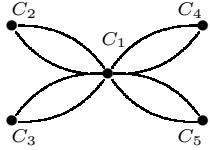
\end{example}

\section{Applications}
If $g\geq 3$ we denote by $\Hgb\subset \Mgb$ the closure of the locus of hyperelliptic curves.
Recall that $\Hgb$ is an irreducible subscheme of dimension $2g-1$.
Following a  common practice (see \cite{HM}), we say that a stable curve $X$ is {\it  hyperelliptic} if $[X]\in \Hgb$.
\begin{defi}
\label{wh} We call a stable curve $X$ {\it weakly hyperelliptic} if there exists a balanced line bundle $L\in \picX{2}$
such that $h^0(X,L)\geq 2$.
\end{defi}

\begin{lemma} 
\label{hwh}
 If $X$ is  hyperelliptic,   $X$ is weakly  hyperelliptic.
\end{lemma}
\begin{remark}
The converse  is  false, see Remark~\ref{whnh}.\end{remark}
\begin{proof}
As $[X]\in \Hgb$ there exists a one parameter smoothing of $X$,
 $f:\X\to \Spec R$, whose generic fiber is 
a smooth hyperelliptic curve. We can also assume that $\X$
is regular, and that there exists $\L\in \Pic \X$ such that
the restriction of $\L$ to the generic fiber is the hyperelliptic bundle.
Set $L=\L_{|X}$.
Up to  tensoring $\L$ with a divisor supported entirely on the closed fiber $X$ we can assume that $L$ is
balanced. By uppersemicontinuity of $h^0$ we have $h^0(X, L)\geq 2$,
so we are done.
\end{proof}

\subsection{Clifford index of   two-components curves} 
Recall that   smooth hyperelliptic  curves can be characterized 
using Clifford's inequality; 
the same holds for irreducible curves (see  \cite[Sec. 5]{Ctheta}).
We shall generalize this to stable curves having two components.
So, let $X=C_1\cup C_2$ have genus $g\geq 2$; we proved in Theorem~\ref{clvine} that the Clifford's inequality holds.

The Clifford index of a line bundle has been introduced in  \ref{index}.  
Now, if $X$ is   irreducible, its Clifford index    is 
 defined as $\Cl X=\min \{\Cl L\}$
where $L $ varies in the set of line bundles on $X$ such that $h^0(X,L)\geq 2$ and $h^1(X,L)\geq 2$. By   Clifford's theorem,
$\Cl X\geq 0$; moreover, $\Cl X=0$ if and only if  $X$ is hyperelliptic.
We extend the definition
of the Clifford index to a semistable curve $X$ as follows.
\begin{equation}
\label{ClX}
\Cl  X =\min\{\Cl L \  |\  \mdeg L\in \BXd, \  h^0(X,L)\geq 2,\  h^1(X,L)\geq 2 \}.
\end{equation}

By  Theorem~ \ref{clvine}, $\Cl X\geq 0$ if $X=C_1\cup C_2$. We now ask: when is $\Cl X=0$?
To answer this question we use the following terminology.
A  curve $X$ (reduced, nodal, of genus $g$) is called  a {\it binary curve} if it is the
 union of two copies of $\pr{1}$ meeting transversally in $g+1$ points (cf. \cite{Cbin}). 
 
\begin{prop}
\label{clwh} Let $X=C_1\cup C_2$ be semistable.
\begin{enumerate}
\item
$\Cl X=0$ if and only if $X$ is weakly hyperelliptic.
\item
If $X$ is weakly hyperelliptic, then $C_1\cdot C_2\leq 2$ unless $X$ is a hyperelliptic binary curve.
\end{enumerate}

\end{prop}
\begin{proof}
As we said,
 Theorem~\ref{clvine} yields   $\Cl  X\geq 0$. Therefore  if $X$ is weakly hyperelliptic, then $\Cl X=0$.

Conversely, suppose $\Cl X=0$; let $L\in \Pic^{\md}(X)$ with $\md \in \BXd$,
such that $h^0(L)=d/2+1$.
If $d=2$ there is nothing to prove, so   assume $d>2$.
As usual, set $\dd =C_1\cdot C_2$.
We  must prove that there exists a   $J\in \picX{2}$ such that $h^0(J)=2$ and $\mdeg J\in \BXt$.

$\bullet$
Assume first $d_i\leq 2g_i$ for $i=1,2$. By Corollary~\ref{ecor} we have $\dd \leq 2$.

{\it Suppose $\dd=2$};  again by Corollary~\ref{ecor}  we have $\Cl L_1=\Cl L_2=0$ 
and, if $d_i\geq 2$, then $\Cl L_i(-C_1\cap C_2)=0 $.

If  $d_1=0$ then  $L_1=\O_{C_1}$ and $L_2=H_2^{d/2}$
for some $H_2\in W^1_2(C_2)$
(see \cite[subsec. 5.2]{Ctheta}). By hypothesis $(0,d)\in \BXd$, which easily implies that $g_2>g_1$,
and hence that  multidegree $(0,2)$ is balanced.
Consider the line bundle $M:=(\O_{C_1},H_2)$ on the normalization $\Xn$ of $X$;
as $\Cl H_2^{d/2}(-C_1\cap C_2)=0 $ we have 
$h^0(C_2,H_2(-C_1\cap C_2))=1$, hence by Lemma~\ref{d1l}
 there exists $J\in \FM$ such that
$h^0(X,J)=h^0(\Xn,M)-1=2$. Since  $\mdeg J=(0,2)$ is balanced we are done.

If   $d_i>0$ for $i=1,2$ then there exists $H_i\in W^1_2(C_i)$ such that  $L_i=H_i^{d_i/2}$, for both $i$.
Suppose $g_1\leq g_2$;  arguing as above we see that $(0,2)$ is balanced and that
there exists $J\in W^1_{(0,2)}(X)$ such that the pull-back of $J$ to the normalization of $X$
is $(\O_{C_1},H_2)$. Up to switching $C_1$ and $C_2$ we are done.

{\it Suppose $\dd=1$.}
If $(1,1)$ is balanced, then $X$ is (trivially) weakly hyperelliptic (see Lemma~\ref{ctwh}).
So assume $(1,1)$ not balanced. By Example~\ref{ct2} we may assume  $g_1<g_2$ and  $B_2(X)=\{(0,2)\}$.
By Corollary~\ref{ecor}, $\Cl L_2=0$, therefore $C_2$ is hyperelliptic. Let $H_{C_2}$ be its hyperelliptic bundle, and set 
$J=(\O_{C_1}, H_2)$; it is clear that $h^0(X,J)=2$.

$\bullet$ Now assume that $d_1= 2g_1+e$ with $e\geq 1$.
We will prove that $X$ is  a binary curve. In this case  the result  is known:   a binary curve is hyperelliptic if and only if it
is weakly hyperelliptic  (\cite[Sec. 3]{Cbin}).

We are in the situation treated in the proof of \ref{clvine},
from which we now use the notation. We saw there that
   the Clifford inequality can be an equality
only in Case 2, at the very end. More precisely,   in order for
$\Cl L=0$ we must have $d_2=2g_2+e$ 
(so that 
$ 
d=2g_1+2g_2+2e$) and 
\begin{equation}
\label{lll}
l=l_1+l_2-e-1.
\end{equation}
Now, as $d<2g-2$ and $g=g_1+g_2+\dd-1$ we have $2(g_1+g_2+e)<2(g_1+g_2+\dd-2)$, hence
\begin{equation}
\label{edd}
e\leq \dd -3.
\end{equation}
Now let $\beta: =e+1$, so  that $\beta \leq \dd -2$. Set 
$$
Y=(C_1\coprod C_2)/_{\{p_i=q_i,\  \  i=1,\ldots,  \beta\}}  \stackrel{\nu}{\la} X,
$$
i.e. $\nu$ is the normalization of $X$ at $\dd-\beta$ nodes.
Let $M=\nu^*L$;
we have, by Lemma~\ref{e} (\ref{e=}), 
$$
h^0(Y,M)=l_1+l_2-e-1=l=h^0(X,L)
$$
using (\ref{lll}). Therefore for all $
i=\beta +1,\ldots, \dd$, we have $p_i\sim_Mq_i$, by Lemma~\ref{d1l}.
This  implies that,  for all $i\geq \beta +1$, 
$p_i$ is a base point of $L_1(-\sum_{j=1}^{\beta}p_j)$ and
$q_i$ is a base point of $L_2(-\sum_{j=1}^{\beta}q_j)$ (by Lemma~\ref{bpl}).
Now 
$$
\deg L_1(-\sum_{j=1}^{\beta}p_j)=2g_1+e-\beta=2g_1-1,\       \  \deg
L_2(-\sum_{j=1}^{\beta}q_j)=2g_2+e-\beta=2g_2-1.
$$
If $X$ is not a binary curve, we may assume
   $g_2\geq 1$. 
Then,  $L_2(-\sum_{j=1}^{\beta}q_j)$, having degree $2g_2-1$, can have at most one
base point. Therefore $\dd -\beta \leq 1$, i.e. $\dd -e \leq 2$, which is in contradiction with (\ref{edd}).
We conclude  that $X$ is a binary curve.
\end{proof}
\begin{nota}{\it Curves  of compact type}
For any integer $h$ with $1\leq h\leq g/2$, let
$\Delta_h$ be the divisor in $\Mgb$ whose general point represents a curve $X=C_1\cup
C_2$ with
$C_i$ smooth,
$C_1\cdot C_2=1$ and $g_1=h$. 
Fix such an $X$; for $i=1,2$ we shall denote by $q_i\in C_i$   the branches of the node of $X$.
We computed $B_2(X)$ in Example~\ref{ct2}.
\begin{lemma}
\label{ctwh} 
Let $X=C_1\cup C_2$ with $C_1\cdot C_2=1$ and $1\leq g_1 \leq g/2$.

Let $g_1\geq (g+1)/4$. Then $X$ is weakly hyperelliptic; more precisely,
$(1,1)$ is balanced  and $W^1_{(1,1)}(X)=\{ (\O_{C_1}(q_1),\O_{C_2}(q_2)\}$.

Let $ g_1<(g+1)/4$.  Then 
 $X$ is weakly hyperelliptic if and only if $C_2$ is hyperelliptic,
if and only if $W^1_{(0,2)}(X)=\{(\O_{C_1},  H_{C_2} )\}$.
\end{lemma}
\begin{proof}
Set $L=(\O_{C_1}(q_1),\O_{C_2}(q_2))\in \Pic X$.
It is clear that $h^0(X,L)=2$.
If $g_1\geq (g+1)/4$,  then $L$ is balanced. Conversely,
let $L'\in W^1_{(1,1)}(X)$;  
by Corollary~\ref{ecor} we have $L'=(\O_{C_1}(q_1),\O_{C_2}(q_2))$, so the first part is proved.

Now suppose $g_1< (g+1)/4$, then $(0,2)$ is the unique balanced multidegree.
If $C_2$ is hyperelliptic,
  the balanced line bundle
$L=(\O_{C_1},  H_{C_2} )\in \Pic X$ has, of course, $h^0(X,L)=2$. So, $X$ is weakly hyperelliptic.
Conversely, if  there exists   $L\in \Pic^{(0,2)}X$ such that $h^0(L)=2$,
we can apply Corollary~\ref{ecor} (we necessarily have $g_2\geq 3$ by hypothesis)
and conclude that $h^0(C_2,L_2)=2$, so we are done.
 \end{proof}
\begin{remark}
\label{whnh}
The previous result shows that there exist (plenty of) weakly hyperelliptic curves that are  not hyperelliptic.
Indeed,  it is well known that  a curve of compact type $X=C_1\cup C_2$ is hyperelliptic if and only if
both $C_1$ and $C_2$ are hyperelliptic, and the two branches, $q_1$  and $q_2$, are Weierstrass points (cf. \cite{CH}  for example).
Also, 
there exist  globally generated balanced line bundles $L\in W^1_2(X)$ which are not limits
of hyperelliptic bundles of smooth curves (indeed $(\O_{C_1},  H_{C_2} )$ is always globally generated).
\end{remark}
\end{nota}

\subsection{Hyperelliptic and weakly hyperelliptic curves.}

The next definition will be used only when $\sep =\emptyset$.
\begin{defi}
\label{bp} A pair $(C,D)$ of (smooth, rational) components of $X$  is
 called    a {\it binary-pair} (or a {\it $\B$-pair} for short) of $X$ 
if $C$ is a separating line of $D^c$ and  $D$ is a separating line of $C^c$.
The subcurve $C\cup D$ will be called a $\B$-{\it subcurve}.
\end{defi}

 \begin{example}
Let $X$ be a binary curve
(defined before Proposition~\ref{clwh}); then its irreducible components form a $\B$-pair.
Also, if $X'=C\cup D\cup  E_1\cup\ldots \cup E_s$ is a semistable curve whose stabilization is a binary curve $X=C\cup
D$, then 
$(C,D)$ is a 
$\B$-pair of $X'$ .
\end{example}
Let  $(C,D)$ be a binary pair of $X$.  Denote $C\cap D=\{n_1,\ldots, n_l\}$, with $l\geq 0$,
and $q_C^i\in C$, $q_D^i\in D$ the two branches of $n_i$. If $C\cup D\neq X$, there is a decomposition
$X=(C\cup D) \cup (Z_1 \coprod\ldots \coprod Z_m)$ where $Z_j$ are connected and $Z_j\cdot C=Z_j\cdot D=1$ for all
$j$. Denote  $p_C^j=C\cap Z_j$ and $p_D^j=D\cap Z_j$. Let $n=l+m$ ($m\geq 0$);
now the ordered $n$-tuples 
\begin{equation}
\label{mark}
G_C:=(q_C^1,\ldots ,q_C^l,p_C^1 \ldots ,p_C^m)\subset C,\  \ G_D:=(q_D^1,\ldots ,q_D^l,p_D^1 \ldots ,p_D^m)\subset D
\end{equation}
give a structure of $n$-marked curve on   $C$ and $D$.
We say that  $(C,D)$ is a {\it   special $\B$-pair} if $(C;G_C)$ and $(D;G_D)$ are isomorphic as $n$-marked
curves.

\begin{thm}
\label{c2=}
Let $X$ be   semistable with $\sep =\emptyset$; let  
 $\md$  be such
that
$|\md| =2$. Assume that $\md$ is balanced, or  that $X$ stable and $\md \geq 0$.
Suppose there exists $L\in \picX{\md}$ with $h^0(X, L)=2$. 

Then $L$ is globally generated,    
and either one of the two cases below  occurs.  
\begin{enumerate}[(1)]
\item
\label{2=1}
$\md=(1,1,0\ldots 0)$ and $(C_1, C_2)$ is a special $\B$-pair of $X$. Also, the restriction of $L$ to $\ov{X\smallsetminus (C_1\cup C_2)}$ is trivial.
\item
\label{2=2}
$\md = (2,0,\ldots,0)$ and, denoting $C_1^c=Z_1\coprod\ldots \coprod Z_m$, with $Z_i$ 
 connected,   $\  \forall i=1\ldots m $ we have 
$$
C_1\cdot Z_i=2,\  \  L_{C_1}\cong \O_{C_1}(C_1\cap Z_i),\    \   L_{C_1^c}\cong \O_{C_1^c} \  \  \text{and}\  \  h^0(C_1,L_{C_1})\geq 2.
$$
\end{enumerate}
Conversely, if  $X$ and   $\md$  satisfy the above properties,   there exists a unique line bundle 
$L\in \picX{\md}$ such that
$W_{\md}^1(X)=\{L\}$.
\end{thm}
\begin{proof}
Assume  there exists $L\in W^1_{\md}(X)$;
by Theorem~\ref{cl4} (\ref{=+}), $\md \geq 0$.
By the same Theorem and its addendum we obtain that $L$ is globally generated.

We fix $C$   a non-disconnecting component of $X$ (Remark~\ref{ord}), and set $Z=C^c$.

\noindent
{\bf Step 1.} {\it Case  $(d_C,d_Z)= (1,1)$.}

Let $D\subset Z$ be the component such that $d_D=1$.
We must prove that $(C,D)$ is a $\B$-pair of $X$.

By contradiction, suppose   $D$ is not a separating line of $Z$; this implies  $h^0(Z,L_Z)\leq 1$
(by  Lemma~\ref{lemmino}).
Let $C\not\cong \pr{1}$, then $h^0(C,L_C)\leq 1$.
So, in order to have $h^0(X,L)=2$ we must have
$h^0(C,L_C)=h^0(Z,L_Z)=1$ and 
 every point
in $Z\cap C\subset C$ must be a base point for $L_C$
(by Lemma~\ref{d1l}). This is impossible, as 
$Z\cdot C\geq 2$ and 
$d_C=1$.
Now let $C\cong \pr{1}$, hence $h^0(C,L_C)= 2$. 
By Lemma~\ref{e} we have
$$
h^0(X,L)\leq h^0(C,L_C)+h^0(Z,L_Z)-2\leq 2+1-2=1;$$ a contradiction.

Therefore $D$ is a separating line of $Z$, and  $h^0(Z,L_Z)=2$.
By Remark~\ref{snr} (\ref{snrC}), $D$ is  a non-disconnecting component of $X$. So, we can switch $C$ with $D$ and, 
 by  the
previous argument, we obtain that
$C$ is  a separating line of $D^c$. 
In other words, $(C,D)$ is a  $\B$-pair of $X$, as stated.

Now,   as $h^0(L)=2$, 
 the restriction of $L$ to the complement of $C\cup D$ is trivial.
Therefore $L$ determines a map $\psi$ to $\pr{1}$ such that $\psi(p_C^j)=\psi(p_D^j)$ for all $j$
(notation as in (\ref{mark})).
Hence $\psi$ induces an isomorphism of the $n$-marked curves $C$, $D$ with the same $n$-marked $\pr{1}$.
This shows that the pair $(C,D)$ is special.

\noindent {\bf Step 2.} {\it Case $(d_C,d_Z)= (2,0)$.}

Now    $X$ must be a stable curve (an exceptional component must have degree $1$).
We must prove that   $L_Z\cong \O_Z$, that   $C\cdot Z=2$ and that, setting $C\cap Z=\{p,q\}\subset C$, we have
$\O_C(p+q)\cong L_C$.
Assume first $C\not\cong \pr{1}$.
So $h^0(C,L_C)\leq 2$ with equality only if $L_C$ has no base point; also, $h^0(Z,L_Z)\leq 1$ with
equality  if and only if $L_Z=\O_Z$ (by Fact~\ref{cl00}).  It is clear that, for $h^0(X,L)=2$, we must have equality in both cases.
Hence $L_Z=\O_Z$.  If $C\cdot Z\geq 3$, by    Lemma~\ref{d2} there exist  three points $p,q,r \in C$
such that 
$$p\sim _{L_C}q\sim _{L_C}r.
$$ 
Now $L_C$ has no base points, hence we get
$$
1=h^0(C,L_C)-1=h^0(C,L_C(-p))=h^0(C,L_C(-p-q-r))
$$ 
which is impossible, as
$\deg L_C(-p-q-r)=-1$.
We thus proved that $C\cdot Z=2$, that  $h^0(C,L_C(-p-q))=1$, i.e. $L_C=\O_C(p+q)$.
The uniqueness of $L$ follows from Lemma~\ref{d1l}.

Now let us prove that $C\not\cong \pr{1}$.
By contradiction, if $C\cong \pr{1}$, then $\dd_C\geq 3$ ($X$ is stable) and $h^0(C,L_C)= 3$. By Lemma~\ref{e} we
obtain 
$$
h^0(X,L)\leq h^0(C,L_C)+h^0(Z,L_Z)-3\leq 3+1-3=1
$$ 
which is impossible.

\noindent
{\bf Step 3.} {\it Case $(d_C,d_Z)= (0,2)$.}

Now $h^0(C, L_C)\leq 1$ with equality  if and only if $L_C=\O_C$.
Suppose $Z$ contains a separating line $E$ such that $d_E\geq 1$; then $E^c$ is connected (by Remark~\ref{snr}(\ref{snrC}));
we may thus replace $C$ with $E$, and be back in the situations treated in the previous steps.

So, we are reduced to assume  $Z$ contains no such separating line.
By  Lemma~\ref{lemmino2}, and because $h^0(X,L)=2$, we have 
$
h^0(Z,L_Z)=2 
$
and $L_C=\O_C$.
Let $D\subset Z$ be an irreducible component with $d_D\geq 1$. Denote
$ 
Y=D^c=Y_1\coprod \ldots \coprod Y_m
$ 
the connected components decomposition.
We have 
 ($\sep =\emptyset$)
\begin{equation}
\label{DY}
D\cdot Y_i\geq 2,\  \  \forall i=1,\ldots, m.
\end{equation}

Assume $d_D=1$. Let $Y_1$ be the connected component   such that $d_{Y_1}=1$; then
$h^0(Y_1,L_{Y_1})\leq 1$ (by  Lemma~\ref{lemmino}, as $Y_1$ contains no separating line having degree $1$). Therefore,
setting $X_1=D\cup Y_1\subset X$,   Lemma~\ref{e} yields
$$
h^0(X_1,L_{X_1})\leq 2+1-2=1.
$$
As $h^0(X,L)\leq h^0(X_1,L_{X_1})$ (by Remark~\ref{ur}) we have a contradiction.

Therefore we must have $d_D=2$, hence $\md_{D^c}=\mo$.
Now, for every $i=1,\ldots, m$ we argue as in Step 2, with $X_i=D\cup Y_i$ playing the role of $X$,  
$D$ playing the role of $C$, and $Y_i$ playing the role of $Z$.
This  shows that  $L$ is unique and that for every $i$, $D$ intersects $Y_i$ in two points $p_i,q_i\in D$,
that $L_D\cong \O_D(p_i+q_i)$ and that $L_{Y_i}\cong \O_{Y_i}$.
This concludes Step 3.

The converse follows easily from Lemma~\ref{d1l}.
The proof is complete.
\end{proof}

\begin{remark}
\label{Bdec}
Let $X$ be a stable curve such that $\sep =\emptyset$. Then $X$ admits a
decomposition (unique up to the order)
$X=A_1\cup \ldots \cup A_{\alpha}$ such that every $A_i$ is either a $\B$-subcurve  
  or an irreducible component of $X$ not part of any $\B$-pair

This follows from the fact that,
by Lemma~\ref{1sp},   every irreducible component of $X$ belongs to at most one $\B$-pair.
\end{remark}

\begin{prop}
\label{hypcomb} 
Let $X$ be a hyperelliptic stable curve such that $\sep =\emptyset$. 
Consider  the decomposition $X=A_1\cup \ldots \cup A_{\alpha}$   defined in  Remark~\ref{Bdec}.
Then for every $i\neq j$   we have either $A_i\cap A_j=\emptyset$, or
$$
A_i\cdot A_j=2 \  \  \text{and}\  \  h^0(A_i, \O_{A_i}(A_i\cap A_j))\geq 2.
$$
\end{prop}
\begin{proof}
We begin as in the proof of Lemma~\ref{hwh}
Let $f:\X\to B$ be a  regular one-parameter smoothing of $X$,
 whose generic fiber  is  hyperelliptic,
 and  let
 $\L\in \Pic \X$ be a balanced line bundle such that the restriction of $\L$ to the generic fiber is the hyperelliptic bundle,
set $\L_{|X}=L$.
Now for every divisor $T\in \Div\X$ supported on $X$,
denote 
$$L_T:=\L\otimes \O_{\X}(T)\otimes \O_X.$$ 
For every $T$ we have $\deg L_T=2$ and, by uppersemicontinuity of $h^0$,    
  $h^0(X,L_T)\geq 2$. 

By assumption $\md =\mdeg L$  is balanced. By Theorem~\ref{c2=} we have
$X=A\cup (Z_1\coprod\ldots \coprod Z_m)$ where $A$ is either an irreducible component, in which case $\md_A=2$, or
a
$\B$-subcurve, in which case $\md_A=(1,1)$. Recall that
$Z_i\cap Z_j =\emptyset$, $Z_i\cdot A =2$, and that
if $A$ is a $\B$-pair then $\mdeg_ AZ_i =(1,1)$. We also have (always by Theorem~\ref{c2=})
  $L_{A^c}=\O_{A^c}$ and $h^0(A,  A\cap Z_i)\geq 2$ for every $i$.
Set $A=A_1$.

Consider $L_T$ with $T=-Z_1$. By what we just said $\mdeg L_T\geq 0$, indeed for   any component
(or subcurve) $C\subset Z_1$
we have $\deg _{C}L_T=
-\deg_CZ_1=
C\cdot A_1\geq 0$; if instead $C\subset Z_1^c$   then $\deg _{C}L_T=0$.   
We can apply thus  Theorem~\ref{c2=} to $L_T$.
Since $\mdeg_{A_1}L_T=\mo$ 
and $\mdeg_{Z_i}L_T=\mo$ if $i\neq 1$,
we derive  that $Z_1$ contains a subcurve $A_2$ with the same properties as $A_1$;
in particular, $A_2$ is either irreducible or a $\B$-subcurve,
and   $A_1\cdot A_2 =2$, because $\deg_{A_2}L_T=2$. 
Therefore $A_1\cap Z_1=A_1\cap A_2$ and $h^0(A_1,  A_1\cap A_2)\geq 2$.
Thus
the part of the statement concerning $A_1$ and $A_2$ is satisfied; so, if $A_2=Z_1$ 
 we turn to $Z_i$ with $i\geq 2$. If instead $A_2\subsetneq Z_1$, we iterate
the procedure
with $A_2$ as the starting component and $T=-W$ with $W$ a connected component of $\ov{Z_1\smallsetminus A_2}$.
Obviously this iteration stops after finitely many steps.
By repeating this argument for every $Z_i$ we are done.\end{proof}

 \subsection{Curves of genus $6$ admitting a $g^2_5$}
\begin{nota}
\label{vnot}
Throughout this subsection we shall consider 
curves $X=C_1\cup C_2$, of genus $6$, such that 
$C_1$ and $C_2$  are smooth,
 of respective genus $g_1$ and $g_2$;
 we set $\dd =C_1\cdot C_2$. 
 For any $L\in \Pic X$ we write $L_i=L_{|C_i}$ and $h^0(L_i)=h^0(C_i,L_i)$.
 We  
fix points $p_1,\ldots, p_\dd\in C_1$
and $q_1,\ldots, q_\dd\in C_2$ so that 
$
X=(C_1\coprod C_2)/_{(p_i=q_i,\  \  i=1,\ldots,  \dd )} 
$   and  set
\begin{equation}
\label{G}
G_1:=\sum _{i=1}^{\dd} p_i,\  \  \  G_2:=\sum _{i=1}^{\dd} q_i .
\end{equation}
 Finally, we set $\mg:=(g_1,g_2)$, and we always assume
$ 
g_1\leq g_2.
$ 
\end{nota}

\begin{thm}
\label{g25} With the above set-up, let $X=C_1\cup C_2$ be  semistable  of genus $6$, and  
let $\md\in B_5(X)$. Assume  there exists a globally generated
$L\in W_{\md}^2(X)$.
Then 
\begin{enumerate}[(I)]
\item
\label{1}
If $\dd=1$, \    $C_2$ is not hyperelliptic  and one of the following cases occurs.
\begin{enumerate}[(a)]
\item
\label{11}
$\mg=(1,5),\    \md=(0,5)$,    
$L_1=\O_{C_1} $,   and $ h^0(L_2)=3$.
\item
\label{12}
$\mg=(2,4)$ or $\mg=(3,3)$,\     $\md=(2,3)$,   and
$h^0(L_1)=h^0(L_2)=2$.
\end{enumerate}
\item
\label{2}
If $\dd =2$  one of the following cases occurs.
\begin{enumerate}[(a)]
\item
\label{ss}
$\mg=(0,5),\   \md=(1,4)$, 
$C_2$   hyperelliptic, $L_2=H_{C_2}^{\otimes 2}.$
\item
\label{14}
$\mg =(1,4), \  \md=(0,5)$, 
 $L_1=\O_{C_1}$, \  $C_2$  not hyperelliptic,  $h^0(L_2)=3$.
\item
\label{22}
$ \mg =(2, 3),  \  \md=(2,3)$,
$L_1=H_{C_1}=\O_{C_1}(G_1)$,   $C_2$ not hyperelliptic, $L_2=\O_{C_2}(G_2+q)$ and $h^0(L_2)=2$.
\item
\label{22}
$\mg =(1,4)$ or $ \mg =(2, 3),  \  \md=(2,3)$,
$L_1=\O_{C_1}(G_1)$,   $C_2$ not hyperelliptic, $L_2=\O_{C_2}(G_2+q)$ and $h^0(L_1)=h^0(L_2)=2$.

\end{enumerate}
\item
\label{3}
If $\dd =3$ then $\mg =(1,3)$ and one of the following cases occurs.
\begin{enumerate}[(a)]
\item
 $\md=(3,2)$,  $L_1=\O_{C_1}(G_1)$, $C_2$ is hyperelliptic, $L_2=H_{C_2}$.
\item
$\md=(0,5)$, $L_1=\O_{C_1} $,   and $ h^0(L_2)=3$. 
\end{enumerate}
\item
\label{4}
If  $\dd =4$, then $\mg=(0,3),\    \md=(1,4)$ and $L_2=K_{C_2}=\O_{C_2}(G_2)$.
\item
\label{6}
If $\dd =6$, then $\mg =(0,1),\    \md=(2,3)$. 
\end{enumerate}
\end{thm}
\begin{remark}
\label{3on}
The cases (\ref{1}) and (\ref{2}), i.e  $\dd\leq 2$, are contained in Propositions \ref{W251} and \ref{W252}, where a more precise statement
is proved.
\end{remark}
\begin{proof}
Our curve $X$  has a priori  $\dd\leq 7$ nodes. The case that $\dd=7$, i.e. $X$ is a binary curve,  is ruled out as follows.
Proposition 12 in \cite{Cbin}  implies $\mdeg L=(2,3)$; by Proposition 19 and Lemma 20 in loc. cit.
the curve $X$ must be hyperelliptic. Therefore the canonical morphism  maps $X$ two-to-one onto a rational normal quintic in $\pr{5}$.
Now we argue as for smooth curves (cf. \cite{ACGH} D-9 p. 41): we have $h^0(X, \omega_X \otimes L^{-1})=3$, hence (as points on a rational normal curve are in general linear position) we easily get $L\cong H_X^{\otimes 2}(p)$ with $p\in X$ a base point of $L$. So $L$ is not globally generated, and we are done.

From now on, by Remark~\ref{3on}, we assume $3\leq \dd \leq 6$.

Pick $\md$ and $L\in W^2_{\md}(X)$ as in the statement.
The fact that $\md$ is balanced means
\begin{equation}
\label{BIg-1}
g_i-1\leq d_i \leq g_i-1+\dd,\  \  \ i=1,2,
\end{equation}
and $d_i= 1$ if   $C_i$ is an exceptional component.

Let us, first of all,  show that $\md \geq 0$.
If $d_1<0$ we must have $\md=(-1,6)$, and $g_1=0$.
We have $h^0(X,L)=h^0(C_2, L_2(-\sum _{i=1}^{\dd} q_i))\leq 2$,
because $\deg L_2(-\sum _{i=1}^{\dd} q_i)=6-\dd$. This contradiction shows that $d_i\geq 0$ for $i=1,2$.

For $i=1,2$ we set $l_i=h^0(C_i, L_i)$ and $e_i:=d_i-2g_i$.
Let
$$ 
\epsilon:=\max\{e_1,e_2,0\}+1 \  \text{ and }\  \beta:=\min\{\epsilon, \delta\}.
$$
From  Addendum~\ref{add2}  we have
\begin{equation}
\label{l3}
h^0(X,L)\leq l_1+l_2-\beta \leq 3.
\end{equation}

\noindent
{\bf Step 1.}
We exclude all the cases for which $l_1+l_2-\beta\leq 2$. This only requires a trivial   checking.
To begin with,   
the following cases are  all excluded:
\begin{equation}
\label{exc6}
 \dd =6,\  \mg=(0,1),   \md\in \{ (0,5),   \  (3,2), \  (4,1), \  (5,0)\}.
\end{equation}
Let us just show how to treat $\md=(0,5)$. We have $l_1= 1$, $l_2=5$,
$\epsilon=e_2+1=4$ and $\beta =\min \{4,6\}=4$. Hence $h^0(X,L)\leq 2$.
All other cases are treated in the same way.
If $\dd =6$, we are   left with $\md=(1,4)$ and
$\md=(2,3)$ (of course $\mg =(0,1)$).

Let $\dd =5$, by the same argument, we exclude 
 \begin{equation}
\label{exc5}
\dd =5,\  \  \mg=(0,2),\  \  \md \in \{(2,3) \  (3,2), \  (4,1), \  (5,0)\} 
\end{equation}
 and  we exclude 
\begin{equation}
\label{exc5a} \dd =5, \  \  \mg=(1,1),\  \  \md\in\{ (0,5),  \  (1,4)\}.
\end{equation}

Let $\dd =4$.  We exclude 
 \begin{equation}
\label{exc4}  
\dd =4,\  \  \mg =(0,3),\  \  \md\in\{(2,3), \  (3,2)\}.
\end{equation}
and
 \begin{equation}
\label{exc4a}  
 \dd =4,\  \  \mg =(1,2),\  \  \md=(4,1).
\end{equation}  \smallskip
Finally, this method applies to 
  exclude 
 \begin{equation}
\label{exc3}  
\dd =3,\  \  \mg =(0,4),\  \  \md=(2,3).
\end{equation}
This finishes the list of cases for which $l_1+l_2-\beta\leq 2$.

From  now on we
always have $l_1+l_2-\beta =3$ (by (\ref{l3})).

\noindent
{\bf Step 2.}
To exclude another group of cases we now use Lemma~\ref{bpl} and its consequence, Lemma~\ref{g}.
Let us begin with case $\dd=6$, hence $\mg=(0,1)$, and $\md=(1,4)$.
In this case $\beta =3$, so that we obviously have
\begin{equation}
\label{bin}
3=\beta <d_2=4<\dd=6.
\end{equation}
Let $X'=(C_1\coprod C_2)/_{\{p_i=q_i,\  \  i=1,\ldots,  3
\}} $, let $\nu:X'\to X$ be the same map as in  Lemma~\ref{g} and let $M=\nu^*L$.
Then $h^0(X',M)=3$ (by Lemma~\ref{e}(\ref{e=}), or by Clifford).
By (\ref{bin}) Lemma~\ref{g} applies, yielding that  $h^0(X,L)<3$, a
contradiction.

$\bullet$ By (\ref{exc6}) if $\dd=6$  the only remaining case is $\md=(2,3)$.   (\ref{6}) is proved.

The previous argument can be repeated
every time we have $\beta<d_i<\dd$ for some $i$, enabling us to exclude the
 following  cases.

$\dd=5$, $\mg =(0,2)$   and $\md=(1,4)$. (Here $2=\beta <d_2=4<\dd=5$.)

$\dd=5$,  $\mg =(1,1)$   and $\md=(2,3)$. (Here $2=\beta <d_2=3<\dd=5$.)

$\dd=4$, $\mg =(1,2)$ and $\md\in \{(2,3,)(3,2)\}$
(If $\md=(2,3)$ then  $1=\beta <d_2=3<\dd=4$; if
 $\md=(3,2)$ then $2=\beta <d_1=3<\dd=4$.)

We shall now exclude the two equal multidegree cases 
$$
\dd=5,\  \mg =(0,2),\   \md=(0,5) \  \  \text{ and }\  \  \dd=4,\  \mg =(1,2), \  \md=(0,5),
$$
with $l_1+l_2=5$. Let $X'=(C_1\coprod C_2)/(p_i=q_i, i=1,2)$ so that $X'$ has two nodes.
Let $L'\in \Pic X'$ be the pull back of $L$. Then $h^0(X',L')=3$, so, for $h^0(X,L)=3$ we must have $q_i\sim _{L'}p_i$
for $i\geq 3$. Now, by Lemma~\ref{bpl}, this implies that  $L_2(-q_1-q_2)$ has at least two base points, which is clearly  impossible.

$\bullet$ By Step 2,  (\ref{exc5}) and (\ref{exc5a}) there are no more cases  with $\dd =5$.

\noindent
{\bf Step 3.}
Now we shall use Corollary~\ref{cd} to exclude all the cases for which $l_1+l_2=4$ and
there is $i\in \{1,2\}$ such that $l_i\geq 2$ and $\dd >\Cl L_i+2$.
This amounts to the following list of cases.

$\dd=4$, $\mg =(0,3)$ and 
 $\md=(0,5)$. $l_2=3$ and $\Cl L_2=1$.

$\dd=4$, $\mg =(1,2)$ and 
 $\md=(1,4)$. $l_2=3$ and $\Cl L_2=0$.

By the previous step and (\ref{exc4}) the only case left with $\dd =4$ is
$\mg=(0,3)$  and  $\md =(1,4)$. 
Now $\beta =2$, therefore (as $l_1+l_2-2=3$ by (\ref{l3})) we have $l_2=3$, i.e. $L_2$ is the canonical bundle of $C_2$.
To prove that $L_2=\O_{C_2}(\sum_1^4q_i)$ it suffices to prove that  $L_2(-q_1-q_2)$  has $q_3$ and $q_4$ as base points
(and note that we are free to permute the $q_i$).
We argue as at the end  of Step 2:
let $X'=(C_1\coprod C_2)/_{(p_i=q_i, i=1,2)}$ and 
let $L'$ be the pull back of $L$ to $X'$. Then $h^0(X',L')=3=h^0(X, L)$, so,  $L_2(-q_1-q_2)$ has $q_3$ and $q_4$ as base points.

$\bullet$ (\ref{4}) is proved.

$\dd=3$, $\mg =(1,3)$. 
We exclude $\md=(1,4)$ (as $l_2=3$ and $\Cl L_2=0$),
and $\md=(2,3)$ (as $l_1=2$ and $\Cl L_1=0$).

$\dd=3$, $\mg =(2,2)$. 
We exclude $\md=(1,4)$ (as $l_2=3$ and $\Cl L_2=0$),
and $\md=(2,3)$ (as $l_1=2$ and $\Cl L_1=0$).

\noindent
{\bf Step 4.} From now on we assume $\dd=3$.

Let $\mg =(2,2)$ and $\md =(2,3)$. Now $l_1+l_2=4$  if and only if $L_2=H_{C_2}(p)$. So $L_2$ has a base point, which is  impossible by hypothesis. By Step 3, there are no more balanced multidegrees to treat, when   $\mg =(2,2)$.

Let $\mg =(0,4)$. 
By (\ref{exc4}) there are two cases to rule out: $\md=(0,5)$ and $\md =(1,4)$

Let $\md=(0,5)$. 
As $l=3$  we have  $l_1+l_2=1+3=4$. It is clear that Lemma~\ref{d2} applies,
giving $q_1\sim_{L_2}q_2\sim_{L_2}q_3$.
Therefore, if $1\leq i\neq j\leq 3$
$$
2=h^0(C_2,L_2(-q_i))=h^0(C_2,L_2(-q_i-q_j))=h^0(C_2,L_2(-q_1-q_2-q_3)).
$$
But then $C_2$ is hyperelliptic ($\deg L_2(-q_1-q_2-q_3)=2$), which implies that $L_2$
has a base point. A contradiction.

Let $\md=(1,4)$. As $\beta =2$ and  $l=3$  we have 
 $l_1+l_2=2+3$, so $C_2$ is hyperelliptic and $L_2=H_{C_2}^{\otimes 2}$.
Consider   $X'=(C_1\coprod C_2)/_{(p_i=q_i,\  i=1,2)}\stackrel{\nu}{\la} X
$  and let $M=\nu^*L$.
Then $h^0(X', M)=3$, therefore $p_3\sim_Mq_3$. By Lemma~\ref{bpl} we obtain that $q_3$ is a base point of
$L_2(-q_1-q_2)$, hence  (permuting the gluing points) $H_{C_2}\neq \O_{C_2}(q_i+q_j)$ for all $i\neq j$.
So, $L_2(-q_1-q_2)=\O_{C_2}(q_1'+q_2')$ where $q_1'$ is conjugate to $q_1$ under the hyperelliptic series, and the same for $q_2', q_2$. But then, as $q_3$ is a base point of $L_2(-q_1-q_2)=\O_{C_2}(q_1'+q_2')$, we get that (say) $q_3=q_1'$, which is a contradiction.

$\bullet$ By Step 3, the remaining cases with $\dd=3$ have $\mg=(1,3)$
and either  $\md=(3,2)$ or $\md=(0,5)$. This is (\ref{3}).
\end{proof}

\begin{lemma}
\label{g} Let  $\dd$ and $\beta$ be two positive integers with $\dd>\beta$.
Consider the partial normalization of $X$ defined as follows
$$
X'=(C_1\coprod C_2)/_{\{p_i=q_i,\  \  i=1,\ldots,  \beta \}}
\stackrel{\nu}{\la} X=(C_1\coprod C_2)/_{\{p_i=q_i,\  \  i=1,\ldots,  \dd\}}.
$$
 For $i=1,2$, pick  $L_i\in \Pic C_i$ and  $M\in \Pic (X')$ 
 such that $M_{|C_i}=L_i$.

If $\beta < \deg L_i <\dd$ for some $i$, then $h^0(X,L)<h^0(X', M)$ 
for every $L\in \FM$.
\end{lemma}

\begin{proof} 
We  argue by contradiction, as follows.
We prove that  if $\beta < \deg L_1$, and if there exists $L\in \FM$ such that
$h^0(X,L)=h^0(X',M)$, then $\deg L_1\geq \dd$.

Let  such an $L$  be fixed. 
By  Lemma~\ref{d1l} we have $p_{i} \sim _{M}q_{i}$ for all $i=\beta +1,\ldots, \dd$.
Now  Lemma~\ref{bpl} yields that,   for all $ i\geq \beta +1$,
$p_i$ is a base point of $L_1(-\sum_{j=1}^{\beta}p_j)$.

As $\deg L_1>\beta$,   $\deg L_1(-\sum_{j=1}^{\beta}p_j) \geq 1$.
Now, a line bundle of positive degree can have at most as many base points as its degree. 
We  just proved that $ L_1(-\sum_{j=1}^{\beta}p_j)$ has $\dd-\beta$ base points, hence  
$\deg L_1-\beta\geq \dd -\beta$, i.e. $\deg L_1 \geq \dd$. We are done.
\end{proof}

\begin{prop}
\label{W251} With the set up of \ref{vnot}, let $X=C_1\cup C_2$ be  semistable   of genus $6$,
with $C_1\cdot C_2=1$, and
let $\md\in B_5(X)$. 

There exists a globally generated $L\in W_{\md}^2(X)$
if and only if $C_2$ is not  hyperelliptic and one of the following cases occurs.
\begin{enumerate}[(1)]
\item
\label{W1}
$\mg=(1,5),$
\  
$ \md=(0,5)$,   and 
$L=(\O_{C_1},L_2)$ for some $L_2\in W^2_5(C_2)$.
\item
\label{W23}
$\mg=(2,4) $ or   $\mg=(3,3)$,
\  
$\md=(2,3)$,  $C_1$ is hyperelliptic and
$L=(H_{C_1},L_2)$ for some $L_2\in W^1_3(C_2)$.
\end{enumerate}
\end{prop}

\begin{proof}
As $X$ is semistable we have $g_1\geq 1$.
If $L$ is globally generated,  so are $L_1$ and $L_2$; hence
if  $h^0(X,L)=3$ we have $3= l_1+l_2-1$
by Lemma~\ref{d1l}. Therefore $l_1+l_2=4$.

Case $\mg=(1,5)$. The balanced multidegrees are $(0,5)$ and $(1,4)$.
If $\md =(1,4)$ and $l_1=1$ then
 $L_1$ has a base point, which is not possible. If $l_1=0$ 
then  $h^0(X,L)\leq 2$.
 So  $\md =(1,4)$  is ruled
out.

Assume  $\md =(0,5)$. By the initial observation, we must have $L_1=\O_{C_1}$, $l_2=3$ 
and $L_2$ free from base
points, hence $C_2$ is not hyperelliptic. Conversely,
if $L_2\in W^2_5(C_2)$ then $L_2$ is globally generated, because $C_2$ is not hyperelliptic; let $L=(\O_{C_1},L_2)$
then obviously $h^0(X,L)=3$.

Case $\mg=(2,4)$.
The balanced multidegrees are $(1,4)$ and $(2,3)$.
We rule out $\md =(1,4)$ just as in the previous case.
Assume $\md =(2,3)$; 
as $l_i\leq 2$ we have $l_1=l_2=2$ and $C_2$ cannot be hyperelliptic (for otherwise $L_2$ has a base point).
The converse is easily proved as before.

Case $\mg=(3,3)$. This case is symmetric, so it suffices to consider 
the balanced multidegree  $\md =(2,3)$. We will show that $C_1$ is hyperelliptic and that $C_2$ is not.
If $C_1$ is not hyperelliptic, then    $l_1\leq 1$; as $l_2\leq 2$
to have  $h^0(X,L)=3$ both $L_1$ and $L_2$ must have a base point at the attaching point, which is not possible.
So $C_1$ must be hyperelliptic.  
The rest of the argument is    exactly as in the
previous case.
\end{proof}

\begin{prop}
\label{W252} With the notations of \ref{vnot}, let $X=C_1\cup C_2$ be   of genus $6$ 
with  $C_1\cdot C_2=2$, and
let $\md\in B_5(X)$. 
There exists a globally generated
$L\in W_{\md}^2(X)$ if and only if  one of the following cases occurs.
\begin{enumerate}[(1)]
\item
\label{W20}
$\mg=(0,5)$, 
  \  $\md=(1,4)$,\   $C_2$   hyperelliptic and 
$L_2=H_{C_2}^{\otimes 2}$.
\item
\label{W14}
$\mg=(1,4)$, \  
$  \md=(0,5)$,     $C_2$   non-hyperelliptic, $L_1=\O_{C_1}$,  $h^0(L_2)=3$ and 
 $h^0(L_2(-G_2))=2$.
\item
\label{W14b}
$\mg=(1,4)$,     $\md=(2,3)$, \  
$L_1=\O_{C_1}(G_1)$, \   $C_2$   non-hyperelliptic,  
$L_2=\O_{C_2}(G_2+q)$,  $h^0(L_2)=2$.
\item
\label{W22}
$\mg=(2,3)$,  $\md=(2,3)$, \  
 $H_{C_1}=\O_{C_1}(G_1)=L_1$, \   $C_2$   non-hyperelliptic  and 
$L_2=\O_{C_2}(G_2+q)$,  $h^0(L_2)=2$.
 \end{enumerate}
\end{prop}
\begin{proof} Notice that, as $L$ has no base points, $L_1$ and $L_2$ have no base points.

Let $\mg=(0,5)$  and  $\md=(1,4)$ (this is a strictly semistable curve and $C_1$ its exceptional component).
By Lemma~\ref{e} we have $l\leq l_1+l_2-2\leq 2+3-2=3$, and equality holds  if and only if $l_2=3$,  if and only if $C_2$ is
hyperelliptic and $L_2=H_{C_2}^{\otimes 2}$, as stated.
It is clear that every $L$ pulling back to $(\O(1), H_{C_2}^{\otimes 2})$ on the normalization of $X$ has $h^0(L)=3$.

If $g_1\geq 1$, one checks easily
(by Proposition~\ref{ic} and the fact that $L_1$ and $L_2$ have no base points) that    $l_1+l_2=4$.
Hence by Lemma~\ref{d2} we have 
\begin{equation}
\label{ppqq}
p_1\sim_{L_1} p_2 \text{ and }\   q_1\sim_{L_2}q_2,
\end{equation}
and $L$ is uniquely determined by its pull-back to the normalization, by Lemma~\ref{d1l}.

$\bullet$ Assume  $\mg=(1,4)$. 
If $\md=(0,5)$,
 by Proposition~\ref{ic} (\ref{ic1}) we obtain $L_1=\O_{C_1}$ and  $\Cl L_2=1$ so $h^0(L_2)=3$.
$C_2$ cannot be hyperelliptic, for otherwise $L_2$ will have a base point.
Moreover, as $q_1\sim_{L_2}q_2$, we have
$$
h^0(L_2(-q_1-q_2))=h^0(L_2(-q_1))=h^0(L_2(-q_2))=2.
$$
as claimed. The converse follows easily from Lemma~\ref{d2}.
Suppose now $\md=(1,4)$.
As   $p_1\sim_{L_1}p_2$, we have $L_1=\O_{C_1}(p)$ with $p\neq p_i$.
So, $L_1$ has a base point in $p$,  which is not possible.
This case does not occur.
Finally, let  $\md=(2,3)$. 
We must have $l_1=l_2=2$ (as $C_2$ cannot be hyperelliptic, as before).
By (\ref{ppqq}) we obtain $L_1=\O_{C_1}(p_1+p_2)$ and 
 $L_2=\O_{C_1}(q_1+q_2+q)$ for a (uniquely determined) $q\in C_2$. The converse follows from  Lemma~\ref{d2}.

$\bullet$ Now assume $\mg=(2,3)$.
If $\md =(2,3)$  we argue exactly as in the previous case ($\mg=(1,4)$,\  $\md =(2,3)$).
 If  $\md =(1,4)$ we have  
 $l_1=1$ so that $L_1=\O_{C_1}(p)$ with $p\neq p_i$ for $i=1,2$ 
(as $p_1\sim_{L_1} p_2$). 
So $L$ has a base point in $p$;  this case is excluded.
Finally, if $\md =(3,2)$, arguing as before one obtains that $L_1$ has a base point in $p\in C_1$, impossible.
This finishes all the possible cases, so  we are done.
\end{proof}

\end{document}